\documentclass[11pt]{article}
\usepackage{amsthm, amssymb,mathtools}
\usepackage{dsfont,amssymb,tikz}
\usepackage{mathrsfs,comment}
\usepackage[dvipsnames]{xcolor}
\usepackage[normalem]{ulem}
\usepackage{url}
\usepackage[all,arc,2cell]{xy}
\UseAllTwocells
\usepackage{enumerate}
\usepackage{enumitem}

\newlist{mylist}{enumerate}{1} 
\setlist[mylist,1]{ 
  label=(\arabic*), 
  ref=(\arabic*)    
}
\newlist{myotherlist}{enumerate}{1}
\setlist[myotherlist,1]{
  label=(\Alph*),   
  ref=(\Alph*)      
}

\usepackage[numbers]{natbib}  
\usepackage{hyperref}         
\hypersetup{%
  colorlinks=true,%
  linkcolor=blue,%
  citecolor=blue,%
  filecolor=blue,%
  menucolor=blue,%
  urlcolor=blue,%
  pdfnewwindow=true,%
  pdfstartview=FitBH}

\def\makeautorefname#1#2{\expandafter\def\csname#1autorefname\endcsname{#2}}

\def\equationautorefname~#1\null{(#1)\null}
\makeautorefname{footnote}{footnote}%
\makeautorefname{item}{item}%
\makeautorefname{section}{Section}%
\makeautorefname{subsection}{Section}%
\makeautorefname{theorem}{Theorem}%
\makeautorefname{thm}{Theorem}%
\makeautorefname{cor}{Corollary}%
\makeautorefname{lem}{Lemma}%
\makeautorefname{prop}{Proposition}%
\makeautorefname{def}{Definition}%
\makeautorefname{rem}{Remark}%
\makeautorefname{exmp}{Example}%

\newtheorem{thm}{Theorem}[section]
\newtheorem{cor}{Corollary}[section]
\newtheorem{prop}{Proposition}[section]
\newtheorem{lem}{Lemma}[section]

\theoremstyle{definition}
\newtheorem{defn}{Definition}[section]
\newtheorem{exmp}{Example}[section]
\newtheorem{rem}{Remark}[section]

\makeatletter
\let\c@cor=\c@thm
\let\c@prop=\c@thm
\let\c@lem=\c@thm
\let\c@defn=\c@thm
\let\c@exmp=\c@thm
\let\c@ctex=\c@thm
\let\c@rem=\c@thm
\numberwithin{equation}{section}
\makeatother

\title{Conditions for Equivalence of Random Interlacements and Random Walk Reflected off of Infinity}

\author{
   Yao Yu\\
   Peking University
}

\begin{document}

\maketitle
\begin{abstract}
On a transient weighted graph, there are two models of random walk which continue after reaching infinity: random interlacements, and random walk reflected off of infinity, recently introduced in \cite{gs-reflected-rw}. 
We prove these two models are equivalent if and only if all harmonic functions of the underlying graph with finite Dirichlet energy are constant functions, or equivalently, the free and wired spanning forests coincide. In particular, examples where the models are equivalent include $\mathbb{Z}^d$, cartesian products, and many Cayley graphs, while examples that fail the condition include all transient trees. 
\end{abstract}

\section{Introduction}

Consider a locally-finite, connected, transient graph $\mathcal{G}=(\mathcal{VG},\mathcal{EG})$  equipped with conductance function $\mathfrak{c}:\mathcal{EG}\to(0,\infty)$ on the (unoriented) edges of $\mathcal{G}$. For $ x, y \in \mathcal{VG} $, $ x \sim y $ denotes adjacency. The associated transient random walk has a stationary measure $\pi(x) = \sum_{y \sim x} \mathfrak{c}(x, y)<\infty$ and transition kernel $p(x,y)=\mathfrak{c}(x, y)/\pi(x)$.

Two models can be defined based on this transient random walk. The first is \textbf{random interlacements} (abbreviated RI henceforth), originally introduced in \cite{sznitman-interlacement} and extended to transient weighted graphs in \cite{teixeira-interlacement}. Generally speaking, RI is a Poissonian collection of excursions away from $\infty$ with time order, and the corresponding Poisson Point Process measure is defined using the probability measure of transient random walk on $\mathcal{G}$ (see Definition \ref{RI def}). 

The second model is \textbf{random walk reflected off of infinity} (abbreviated RW$^\infty$ henceforth), introduced in \cite{gs-reflected-rw}. This model is a continuous-time process $X^\infty: [0,\infty)\rightarrow \mathcal{VG}\cup\{\infty\}$ that starts from a vertex in $\mathcal{G}$ and  waits at each $x$ for an exponential time with parameter $\mathfrak{m}(x)$, where $\mathfrak{m}:\mathcal{VG}\to(0,\infty)$ is a sufficiently large function (see Theorem \ref{RW^infty props}). The constraints on $\mathfrak{m}$ accelerates RW$^\infty$ as it approaches $\infty$, causing it to hit and bounce off $\infty$ infinitely often, which leads to its recurrence. Let $\tau_\infty$ be the first hitting time of $\infty$ and $\mathbb{P}^\infty_x$ the law of $\{X^\infty_t\}$ starting from $x\in\mathcal{VG}$, then $\mathbb{P}^\infty_x$-a.s. $\{X_t^\infty\}_{t\ge\tau_\infty}$ comprises infinitely many time-ordered excursions away from $\infty$.

These two models are not equivalent in general, because they treat infinity differently. Note that a graph can have multiple ``points at infinity''; a tree, for instance. When facing this situation, the behavior of RW$^\infty$ after it hits $\infty$ depends on which ``point at $\infty$'' it hits (see Proposition 5.3 in \cite{gs-reflected-rw}). Whereas for RI, different ``points at infinity'' are viewed as wired to one point (see Definition \ref{RI def}). Due to this fundamental difference, one only expects them to be equivalent if there is only ``one point at infinity'' in an appropriate sense. Note that the notion of ends in graph theory is not a sufficiently refined notion of ``points at infinity'' (see Remark \ref{end(x)}). 

To bridge the two models, we construct a stochastic process $\{Z_t\}$, called the \textbf{process version of random interlacements} (see Definition~\ref{Z_t def}). Heuristically, $\{Z_t\}$ is obtained by assigning an exponentially distributed holding time with parameter $\mathfrak{m}(x)$ at each visit to vertex $x$ (see Remark~\ref{remark m_* m_**}), and then concatenating all such segments in the (vertex version of) interlacement ordering introduced in \cite{hutchcroft2018interlacements} (see Definition \ref{T(m,n) def} for details). Let $\mathbb{P}^{\mathrm{RI}}$ be the law of $\{Z_t\}$.

\begin{defn}
    Let $(\mathcal{G}, \mathfrak{c})$ be a weighted graph and $\mathfrak{m}: \mathcal{VG}\to(0,\infty)$ be a function such that RI and RW$^\infty$ on this graph can be defined. We say that these two models are equivalent if
    \begin{align*}
        \forall x \in\mathcal{VG},\quad \left\{X_t^\infty\right\}_{t\ge\tau_\infty} \text{ under } \mathbb{P}^\infty_x \text{ and } \left\{Z_t\right\}_{t\ge0} \text{ under } \mathbb{P}^\mathrm{RI} \text{ are identically distributed}.
    \end{align*}
\end{defn}
\begin{rem}\label{equivalence no m}
    The equivalence of $\{X^\infty_t\}_{t\ge\tau_\infty}$ and $\{Z_t\}_{t\ge0}$ depends only on the underlying weighted graph $(\mathcal{G},\mathfrak{c})$, and not on the specific choice of holding time parameter $\mathfrak{m}$. This follows because, given any valid $\mathfrak{m}_0$ establishing equivalence, the property holds for all such $\mathfrak{m}$ under which RI and RW$^\infty$ can be defined (see Remark \ref{remark m_* m_**}). Consequently, we treat model equivalence as a property of $(\mathcal{G},\mathfrak{c})$.
\end{rem}

In this paper we establish necessary and sufficient conditions for equivalence of the two models. Hereafter, the $\mathfrak{c}$-free spanning forest and the $\mathfrak{c}$-wired spanning forest are abbreviated to $\mathfrak{c}$-FSF and $\mathfrak{c}$-WSF, respectively. The $\mathfrak{c}$-FSF (resp.\ $\mathfrak{c}$-WSF) is the limit of $\mathfrak{c}$-weighted spanning forests on large finite subgraphs of $\mathcal{G}$ with free (resp.\ wired) boundary condition (see Definition \ref{FSF WSF def}), which can be constructed from RW$^\infty$ (resp.\ RI) via the Aldous-Broder algorithm (see Theorem 1.14 in \cite{gs-reflected-rw} and Theorem 1.1 in \cite{hutchcroft2018interlacements}).

\begin{defn}
    A function $f:\mathcal{VG}\to\mathbb{R}$ is called discrete harmonic if $f(x)=\sum_{y\sim x}\mathfrak{c}(x,y)f(y)$ for all $x$, and its Dirichlet energy is defined as $\sum_{(u,v)\in\mathcal{EG}} \mathfrak{c}(u,v)[f(u)-f(v)]^2$. The space $\mathbf{HD}(\mathcal{G})$ denotes the set of all harmonic functions with finite Dirichlet energy, while $\mathbb{R}$ denotes the constant functions. 
\end{defn}

\begin{thm}\label{main}
    For a graph $(\mathcal{G},\mathfrak{c})$, the following are equivalent.
    \begin{myotherlist}[leftmargin=*, labelwidth=1em, itemindent=0pt, align=left]
        \item $\mathbf{HD}(\mathcal{G})=\mathbb{R}$. \label{HD = R}
        \item RI and RW$^\infty$ are equivalent.\label{EQ}
        \item $\mathfrak{c}$-FSF$\ \overset{d}{=}\ \mathfrak{c}$-WSF. \label{F = W} 
    \end{myotherlist}
\end{thm}

Equivalence between \ref{F = W} and \ref{HD = R} has been proven in \cite{lyons-peres} (see Theorem \ref{equis for F=W}), so the proof of this theorem primarily lies in establishing $\ref{HD = R}\Rightarrow\ref{EQ}\Rightarrow\ref{F = W}$. Regarding the three conditions, positive examples include $\mathbb{Z}^d$, cartesian products, and many Cayley graphs, while negative examples include transient trees. 

\begin{rem}\label{P^inf_inf rmk}
    By Theorem \ref{RW^infty props} (iv) and Lemma \ref{VG SM} (Lemma 3.14 in \cite{gs-reflected-rw}), RW$^\infty$ has the Markov and the strong Markov properties at any stopping time $\tau$ for which $X_\tau^\infty$ is a.s. a vertex of $\mathcal{G}$. But in general it is not Markov at $\infty$ (see Proposition 5.3 in \cite{gs-reflected-rw}). In this paper, under assumption \ref{HD = R} we prove a restricted version of the strong Markov property of RW$^\infty$ at $\infty$ (see Lemma \ref{strong Markov}), so we can rigorously define $\{X^\infty_t\}$ started from $\infty$. Intuitively speaking, it says that the behavior of the walk after first hitting $\infty$ is independent of the starting point and the behavior of the process before it hits $\infty$. Let the corresponding probability measure be $\mathbb{P}^\infty_\infty$, and then \ref{EQ} can be restated as: 
   \begin{align*}
        \left\{X_t^\infty\right\}_{t\ge0} \text{ under } \mathbb{P}^\infty_\infty \text{ and } \left\{Z_t\right\}_{t\ge0} \text{ under } \mathbb{P}^\mathrm{RI} \text{ are identically distributed}.
    \end{align*}
\end{rem}

Since we can prove that the process version of RI is always strong Markov (Lemma \ref{Z SM}), a direct application of Theorem \ref{main} will be:
\begin{cor}\label{SM}
    For a graph $(\mathcal{G},\mathfrak{c})$ that satisfies any (and hence all) of the conditions in Theorem \ref{main}, RW$^\infty$ on this graph has the strong Markov property. 
\end{cor}

The rest of the paper is organized as follows. 
Section 2 provides detailed descriptions of the two models and give a rigorous definition of the process version of RI. 
Section 3 is devoted to prove $\ref{EQ}\Rightarrow\ref{F = W}$. The key argument applies the Aldous–Broder algorithm: since RI and RW$^\infty$ are equivalent under \ref{EQ}, they must generate the same spanning forest law, which gives \ref{F = W}. 
Section 4 cites previous work on \ref{F = W}$\Rightarrow$\ref{HD = R} and gives proof of \ref{HD = R}$\Rightarrow$\ref{EQ}. The latter consists of two key steps: in Section 4.2, we use an electrical network construction to show that, under assumption \ref{HD = R}, both models induce the same hitting distribution on any finite set $K\subset\mathcal{VG}$; in Section 4.3 we establish Markov-type properties to obtain excursion independence. The electrical network construction is the most conceptually novel part of the proof. In Section 4.4 we combine this independence with the excursion similarity to deduce the equivalence of the truncated versions of the two models. The final result is then obtained by passing to the limit. In Section 4.5 we prove Corollary \ref{SM}. 
Finally, Section 5 illustrates Theorem \ref{main} with several examples.

\section{Preliminaries}

Let $\left\{X_n\right\}_{n\ge 0}$ be the random walk on $(\mathcal{G},\mathfrak{c})$, with transition kernel $p(x,y)=\mathfrak{c}(x, y)/\pi(x) = \mathfrak{c}(y, x)/\pi(x)$. Recall $(\mathcal{G},\mathfrak{c})$ satisfies that random walk on $\mathcal{G}$ is transient. We use $\mathbb{P}^{\mathrm{RW}}_x$ for the probability measure of this random walk starting from $x$, and we define $\tau_y=\min\{n\ge 0\mid X_n=y\}$.

\subsection{Random interlacements on \texorpdfstring{$\mathcal{G}$}{G}}
We refer to \cite{drs-interlacement-book} for more indepth introduction. 

Let $W_+$ and $W$ be the sets of infinite and doubly infinite paths in $\mathcal{VG}$, endowed with Borel sigma-algebra $\mathcal{W}$ and $\mathcal{W}_+$, respectively. 
Define an equivalence relation on $W$ by: 
\[
    \forall w,w'\in W, w\sim_{\mathrm{time}} w' \iff \exists k\in \mathbb{Z}, w(n)=w'(n+k) \text{ for all } n\in\mathbb{Z}.
\]
Let $W^*=W/\sim_{\mathrm{time}}$ be the set of doubly infinite trajectories modulo time shift, interpreted as excursions that comes from $\infty$ and goes to $\infty$. Denote by $\pi^*:W\to W^*$ the canonical projection.
The sigma-algebra $\mathcal{W}^*$ on $W^*$ is defined as:
\begin{equation}\label{W^* def}
A \in \mathcal{W}^* \iff (\pi^*)^{-1}(A) \in \mathcal{W}.
\end{equation}
In particular, for any finite $K\subset \mathcal{VG}$, define: 
\begin{align}\label{W_K^* def}
    W_K = \{ w \in W : X_n(w) \in K \text{ for some } n \in \mathbb{Z} \} \in\mathcal{W}, 
    \quad W^*_K=\pi^*(W_K).
\end{align}

\begin{defn}\label{capacity def}
    For any finite $K\subset \mathcal{VG}$ and $y\in K$, the equilibrium measure on $K$ is
    \[
        \mathrm{e}_K(y)=\pi(y)\mathbb{P}^{\mathrm{RW}}_y[\widetilde{H}_K=\infty], \quad \forall y\in K,
    \]where stopping time $\widetilde{H}_K=\inf\left\{n\ge 1\mid X_n\in K\right\}$. And the normalized equilibrium measure is
    \[
        \tilde{\mathrm{e}}_K(y)=\frac{\mathrm{e}_K(y)}{\mathrm{cap}(K)},\ \text{where }\mathrm{cap}(K)=\sum_{y\in K} \mathrm{e}_K(y).
    \]
\end{defn}

Now we can define a measure $\mathbb{Q}_K$ on $(W, \mathcal{W})$. For any $A, B \in \mathcal{W}_+$ and $y \in K$,
\[
    \mathbb{Q}_K[(X_{-n})_{n \geq 0} \in A, \, X_0 = y, \, (X_n)_{n \geq 0} \in B] = \mathbb{P}^{\mathrm{RW}}_y[A\mid\widetilde{H}_K = \infty ]\cdot \mathrm{e}_K(y) \cdot \mathbb{P}^{\mathrm{RW}}_x[B].
\]

\begin{thm}[Theorem 2.1 in \cite{teixeira-interlacement}]
    There exists a unique sigma-finite measure $\nu$ on $(W^*, \mathcal{W}^*)$ which satisfies for all finite $K \subset \mathcal{VG}$,
    \[
        \forall A \in \mathcal{W}^*, A \subset W^*_K: \quad \nu(A) = \mathbb{Q}_K[(\pi^*)^{-1}(A)].
    \]
    Here $\mathcal{W}^*$ and $W^*_K$ are defined in \eqref{W^* def} and \eqref{W_K^* def}.
\end{thm}

\begin{defn}\label{RI def}
    Let $\lambda$ be the Lebesgue measure on $\mathbb{R}$. Let 
    $(\Omega, \mathcal{A}, \mathbb{P}^\mathrm{RI})$ be the probability space where:
    \begin{itemize}
        \item $\Omega \text{ is the set of locally finite point measures where } \boldsymbol{\omega} := \sum_{n \geq 0} \delta_{(w_n^*, u_n)} \text{ on } W^* \times \mathbb{R}_+$,  
        such that $\boldsymbol{\omega}(W_K^* \times [0, u]) < \infty 
        \text{ for any finite } K \subset \mathcal{VG}, u \geq 0$;
        \item $\mathcal{A}$ is the sigma-algebra generated by the evaluation maps $\boldsymbol{\omega} \mapsto \boldsymbol{\omega}(D), \  D \in \mathcal{W}^* \otimes \mathcal{B}(\mathbb{R}_+)$; 
        \item $\mathbb{P}^\mathrm{RI}$ is the probability measure under which $\boldsymbol{\omega}$ is the Poisson Point Process with intensity $\nu \otimes \lambda$.
    \end{itemize}
    The random element of $(\Omega, \mathcal{A}, \mathbb{P}^\mathrm{RI})$ is called the \textbf{random interlacement point process}. 
\end{defn}

In what follows, we will consider the probability measure of RI and conditionally independent Exp$(1)$ random variables $\left\{\tilde{t}(m,n)\right\}_{m\ge 0, n\in\mathbb{Z}}$ conditioned on RI (see Definition \ref{T(m,n) def}). We slightly abuse the notation by using $\mathbb{P}^\mathrm{RI}$ for this probability measure as well.

\subsection{Process version of RI}
We now define a continuous-time stochastic process $\{Z_t\}$ based on RI. 
Heuristically, we assign an exponential waiting time (governed by $\mathfrak{m}$) to each vertex visited in every excursion. These time segments are then concatenated according to a certain order, which is decided by the sequential order of vertices within each single excursion and the temporal order between different excursions. In \cite{hutchcroft2018interlacements}, this order is denoted as the interlacement ordering. 

Since $\mathcal{G}$ is locally finite, let $\{\mathcal{G}_n\}_{n\ge 1}$ be an increasing family of finite, connected subgraphs such that $\mathcal{G}=\cup_{n \geq 1}\mathcal{G}_n$. We first show that RI viewed as a stochastic process has the strong Markov property. With a slight abuse of notation, we also use $\mathbb{P}^\mathrm{RI}$ to denote the probability governing the process $R$ defined below. 

\begin{lem}\label{R SM}
    For RI, consider $\boldsymbol{\omega} = \sum_{m \geq 0} \delta_{(w_m^*, u_m)}\in \Omega$. Define a stochastic process by $R_t=\sum_{u_m\le t}\delta_{(w_m^*,u_m)}$ and define $\mathcal{F}_t$ as the sigma-algebra generated by $\{R_s\}_{s\le t}$. Then $R$ has the strong Markov property, i.e., for any finite stopping time $\eta$, $\{\mathbb{E}^\mathrm{RI}[R_{t+\eta}-R_\eta\mid\mathcal{F}_\eta]\}_{t\ge 0}$ and $\{R_t\}_{t\ge 0}$ coincide in law. 
\end{lem}
\begin{proof}
Recall the definition of $W_{\mathcal{VG}_n}^*$ in \eqref{W^* def}. Let $R^n_t=\sum_{u_m\le t, w_m^*\in W_{\mathcal{VG}_n}}\delta_{(w_m^*,u_m)}$ and let $\mathcal{F}^n_\eta$ be the sigma-algebra generated by $\{R_s^n\}_{s\le t}$. 
Since $\mathrm{cap}(\mathcal{VG}_n)<\infty$, $R^n_t$ is a compound Poisson process and hence right-continuous; therefore, it has the strong Markov property. Additionally, $R-R^n$ corresponds to $W^*\setminus W^*_{\mathcal{VG}_n}\times [0,\infty)$ means it is independent of $R^n$, so we have

\begin{equation}\label{R^n}
    \{\mathbb{E}^\mathrm{RI}[R^n_{\eta+t}-R^n_\eta\mid\mathcal{F}_\eta]\}_{t\ge0}
    \overset{d}{=}
    \{\mathbb{E}^\mathrm{RI}[R^n_{\eta+t}-R^n_\eta\mid\mathcal{F}^n_\eta]\}_{t\ge0}
    \overset{d}{=}\{R^n_t\}_{t\ge0}. 
\end{equation}

Let $\mathcal{M}$ be the space of all at most countable sums of Dirac measures on $W^*\times\mathbb{R}_+$. For any $\mu\in\mathcal{M}$, define the induced measure on $\mathcal{VG}_n$ as $\mu^n(dx)=\mathbf{1}_{x\in W^*_{\mathcal{VG}_n}\times\mathbb{R}_+}\mu(dx)$, and then we can define a metric on $\mathcal{M}$ as:
\[
    d_{\mathcal{M}}(\mu_1,\mu_2):=\sum_n 2^{-n} \cdot\mathbf{1}_{\mu_1^n=\mu^n_2}\le 1. 
\]
It is well-defined, nonnegative, symmetric, and has triangular inequality. 
Then for all $t$, almost surely
\begin{equation}\label{d(Rn,R)=0}
    d_{\mathcal{M}}(R^n_t,R_t)\le \sum_{k>n} 2^{-k}=2^{-n}\to 0,
    \quad n\to\infty,
\end{equation}
\begin{equation}\label{d(ERn,ER)=0}
    d_{\mathcal{M}}\big(\ \mathbb{E}^\mathrm{RI}[R^n_{\eta+t}-R^n_\eta\mid\mathcal{F}_\eta], \ \ 
    \mathbb{E}^\mathrm{RI}[R_{\eta+t}-R_\eta\mid\mathcal{F}_\eta]\ \big)
    \le \sum_{k>n} 2^{-k}=2^{-n}\to 0, \quad n\to\infty.
\end{equation}
We further define a metric on all Lebesgue measurable functions $f:[0,\infty)\to \mathcal{M}$ as 
\[
    d_{R}(f,g)=\int_0^\infty e^{-t}d_{\mathcal{M}}(f(t),g(t))dt\le \int_0^\infty e^{-t}dt<\infty.
\]
By Fubini's Theorem, \eqref{d(Rn,R)=0} and \eqref{d(ERn,ER)=0}, as $n\to\infty$, we have that almost surely 
\[
    d_{R}(R^n,R)\to 0, \quad 
    d_{R}\big(\ \big\{\mathbb{E}^\mathrm{RI}[R^n_{\eta+t}-R^n_\eta\mid\mathcal{F}_\eta]\big\}_{t\ge0}, \ \ \big\{\mathbb{E}^\mathrm{RI}[R_{\eta+t}-R_\eta\mid\mathcal{F}_\eta]\big\}_{t\ge0}\ \big) \to 0.
\]
Therefore, by taking $n\to\infty$ in \eqref{R^n}, 
\[
   \{\mathbb{E}^\mathrm{RI}[R_{\eta+t}-R_\eta\mid\mathcal{F}_\eta]\}_{t\ge0}\overset{d}{=}\{R_t\}_{t\ge0}. 
\]
This gives the result.
\end{proof}

Next we start constructing the process version of RI by defining jump times $T(m,n)$ and $t(m,n)$.
\begin{defn}\label{T(m,n) def}
    For RI, consider $\boldsymbol{\omega} = \sum_{m \geq 0} \delta_{(w_m^*, u_m)}\in \Omega$. 
    For any $m$, choose an equivalence class representative $w_m\in \pi^{-1}(w_m^*)$ in a manner which depends only on $w_m^*$.    \footnote{Here is a canonical way of choosing $w_m$: Since $\mathcal{G}$ is countable, we can list all vertices as $v_1,v_2,v_3,..$. For any $w_m^*$, take $j_m=\min\left\{j\mid v_j \text{ is on } w_m^*\right\}$, and then choose $w_m$ such that $w_m(0)$ is the first time of hitting $v_{j_m}$.}
    Recall that $\left\{\tilde{t}(m,n)\right\}_{m\ge 0, n\in\mathbb{Z}}$ are conditionally independent Exp$(1)$ random variables conditioned on RI. Consider the vertex version of interlacement ordering in \cite{hutchcroft2018interlacements} and define the index set of all points visited before $w_m(n)$ by:
    \begin{equation}\label{I(k,j)}
        t(m,n)=[\mathfrak{m}(w_m(n))]^{-1}\tilde{t}(m,n),
        \quad I(m,n)=\left\{(k,j)\mid u_k<u_m\right\}\cup\left\{(m,j)\mid j<n\right\}.
    \end{equation}
    We define time function $T(m,n)=\sum_{(k,j)\in I(m,n)} t(k,j)$. 
\end{defn}

The next lemma, parallel to Lemma 3.5 in \cite{gs-reflected-rw}, shows that this is a reasonable choice of time. 

\begin{lem}\label{m^**}
    There exists a function $\mathfrak{m}_{**} : \mathcal{VG} \to (0,\infty)$ such that if $\mathfrak{m}(x) > \mathfrak{m}_{**}(x)$ for all but finitely many $x \in \mathcal{VG}$, the following statements hold $\mathbb{P}^\mathrm{RI}$-a.s.:
    \begin{mylist}
        \item $T(m,n)<\infty, \quad \forall m\ge 0, n \in \mathbb{Z}$, \label{prop1}
        \item $\sup_{m\ge 0,n\in\mathbb{Z}}\ T(m,n)=\infty$. \label{prop2}
    \end{mylist}
\end{lem}
\begin{proof}
    For \ref{prop2}, take $o\in\mathcal{VG}$ and recall the definition of $W^*_K$ from \eqref{W_K^* def}. 
    
    Consider $\#\left\{k\mid u_k\le u, w_k^*\in W^*_{\{o\}}\right\}\sim\mathrm{Poisson}\ (u\cdot \mathrm{cap}(\{o\}))$, we have $\#\left\{k\mid w_k^*\in W^*_{\{o\}}\right\}=\infty$ $\mathbb{P}^\mathrm{RI}$-a.s. Let $\xi_1,\xi_2,...$ be i.i.d. exponential variables of parameter $\mathfrak{m}(o)$. Then
    \begin{equation*}
        \sup_{m\ge 0,n\in\mathbb{Z}}\ T(m,n)
        \ge \sum_{k=0}^{\infty}\sum_{n\in\mathbb{Z},w_k(n)=o}\ t(k,n)
        \ge \sum_{i=1}^{\#\left\{k\mid w_k^*\in W^*_{\{o\}}\right\}} \xi_i \ =\    \infty\quad a.s.
    \end{equation*}

    For \ref{prop1}, let $m_j$ be the index whose corresponding $u_{m_j}$ is the $j$-th smallest among all corresponding $u_m$ that satisfies $w_{m}^*$ passes through $o\in\mathcal{VG}$. Assume $w_{m_j}$ first hits $o$ at $w_{m_j}(a_j)$. 
    
    Next, we prove $T(m_2,a_2)-T(m_1,a_1)<\infty$ a.s.
    Let 
    \[
        B = \left\{(m,a)\mid m_1<m<m_2,\text{ or } m=m_1 \text{ and } a\ge a_1, \text{ or } m=m_{2} \text{ and } a<a_{2}\right\},
    \]
    \[
        B_n = \left\{(m,a)\in B\mid w_m(a)\in \mathcal{VG}_{n}\setminus\mathcal{VG}_{n-1}\right\}, 
        \ \text{ and }
        S_n=\sum_{(m,a)\in B_n} t(m,a).
    \]    
    Since $S_n$ is a finite sum of independent exponential variables a.s., we can take $c_n>0$ such that if $\min_{x\in \mathcal{VG}_{n}\setminus\mathcal{VG}_{n-1}}\mathfrak{m}(x)=c_n$, we have $\mathbb{P}^\mathrm{RI}[S_n<2^{-n}]>1-2^{-n}$. Now we define $\mathfrak{m}_{**}$ by
    \[
        \mathfrak{m}_{**}(x)=c_n,\ \forall x\in \mathcal{VG}_{n}\setminus\mathcal{VG}_{n-1}.
    \]
    By the Borel-Cantelli Lemma, we have $T(m_2,a_2)-T(m_1,a_1)=\sum_{n=1}^\infty S_n \lesssim\sum_{n=1}^\infty 2^{-n}<\infty$ a.s. Since $\inf\{t\mid R_t(W^*_{\{o\}}\times\mathbb{R}_+)>0\}$ is a stopping time for $R$, by Lemma \ref{R SM}, $\left\{T(m_{j+1}, a_{j+1})-T(m_j,a_j)\right\}_{j=1}^\infty$ is i.i.d. Moreover, let $T_\infty=\lim_{a\to\infty} T(m_1,a)$, and then 
    \[
        T(m_1,a_1)\overset{d}{=}T(m_2,a_2)-T_\infty\le T(m_2,a_2)-T(m_1,a_1)<\infty\quad\text{ a.s}.
    \]
    Hence almost surely $T(m_j,a_j)<\infty$ for any $j$. Now for any $(m,n)$, there exists $j$ such that $u_m<u_{m_j}$, so $T(m,n)<T(m_j,a_j)<\infty$ $\mathbb{P}^\mathrm{RI}$-a.s. This gives \ref{prop1}.
\end{proof}

\begin{defn} (process version of RI)\label{Z_t def}
    Let $\mathfrak{m}_{**}$ be as defined in Lemma \ref{m^**}. For any $\mathfrak{m}$ such that $\mathfrak{m}(x) > \mathfrak{m}_{**}(x)$ for all but finitely many $x \in \mathcal{VG}$, the process version of RI is defined as
    \[
        \forall t\ge0, \quad Z_t:=
        \begin{cases}
             w_m(n),&\quad \text{ if } T(m,n)\le t< T(m,n)+t(m,n)\\
             \infty,& \quad \text{else}.
        \end{cases}
    \]
    We denote the law of $\{Z_t\}$ as $\mathbb{P}^\mathrm{RI}$.
\end{defn}
\begin{rem}
In \cite{eisenbaum2022isomorphism}, the authors introduce the extended Markov process $\{Y_t\}_{t\ge 0}$ for transient Markov process $\{X_t\}_{t\ge 0}$ that satisfies certain conditions, and they prove that the local time process of $\{Y_t\}_{t\ge 0}$ can be related to the occupation times of RI and the Gaussian free field. Here $\{Y_t\}_{t\ge\tau_\infty}$ is the same as $\{Z_t\}_{t\ge0}$ in Definition \ref{Z_t def}. 
However, we do not expect that RW$^\infty$ fits into the framework of \cite{eisenbaum2022isomorphism}, because \cite{eisenbaum2022isomorphism} wires the boundary at infinity to a single point.
\end{rem}

At last, we prove a simple corollary to prepare for our proof of Theorem \ref{main}.
\begin{cor}\label{hat tau infty a.s.}
    Almost surely, for any $x\in\mathcal{VG}\cup\{\infty\}$, there exist arbitrarily large values of $t \geq 0$ for which $Z_t = x$.
\end{cor}
\begin{proof}
    If $x\in \mathcal{VG}$, since $\boldsymbol{\omega}(W^*_{\{x\}}\times[n,n+1])\sim \mathrm{Poisson}(\mathrm{cap}(\{x\}))$ and these random variables are mutually independent for different $n$, we have that $\boldsymbol{\omega}(W^*_{\{x\}}\times[n,n+1])>0$ for infinitely many $n\ \mathbb{P}^\mathrm{RI}$-a.s. By Lemma \ref{m^**}, this gives the proof.
    
    If $x=\infty$, consider $\lim_{n\rightarrow\infty}[T(m,n)+t(m,n)]$ for any $m$. By Lemma \ref{m^**}, this value is finite a.s.     Assume $\lim_{n\rightarrow\infty}[T(m,n)+t(m,n)]\in[T(k,j),T(k,j)+t(k,j))$ for some $(k,j)$, so $k>m$. 
    Thus, 
    \[
        T(k,j)\le\lim_{n\rightarrow\infty}[T(m,n)+t(m,n)]\le \inf_{i\in\mathbb{Z}}T(k,i)<T(k,j), 
    \]
    a contradiction. Therefore, $Z_{\lim_{n\rightarrow\infty}[T(m,n)+t(m,n)]}=\infty$. By Lemma \ref{m^**}, this gives the proof.
\end{proof}

\subsection{Random walk reflected off of infinity on \texorpdfstring{$\mathcal{G}$}{G}}

The following proposition defines a useful class of functions, which we use to describe the RW$^\infty$ model.
\begin{prop}[Proposition 1.2 in \cite{gs-reflected-rw}]\label{h^y def}
Let $A\subset \mathcal{VG}$ be a non-empty finite set. For each $\phi: A \to \mathbb{R}$, there exists a unique function $h^\phi : \mathcal{VG} \to \mathbb{R}$ such that:
\begin{mylist}
    \item $h^\phi|_A = \phi$ (boundary condition),
    \item $\mathrm{Energy}(h^\phi)$ is minimal among all functions $f : VG \to R$ with $f|_A = \phi$.
\end{mylist}
This $h^\phi$ is discrete harmonic on $VG \setminus A$, has finite Dirichlet energy, and the mapping $\phi \mapsto h^\phi$ is linear.
\end{prop}
For $y \in A$, let $\mathds{1}_y : A \to \mathbb{R}$ be the indicator function which equals 1 at $y$ and zero elsewhere. We write $h^{\mathds{1}_y}_A(x)$ as $h^y_A(x)$ for short. 

With the previous proposition, we are ready to introduce RW$^\infty$. 
\begin{thm}[Theorem 1.5 in \cite{gs-reflected-rw}]\label{RW^infty props}
There exists a function $\mathfrak{m}_* : \mathcal{VG} \to (0, \infty)$ such that for each function $\mathfrak{m} : \mathcal{VG} \to (0, \infty)$ with $\mathfrak{m}(x) \geq  \mathfrak{m}_*(x)$, $\forall x \in \mathcal{VG}$ and for each possible starting vertex $z \in \mathcal{VG}$, the following is true. There is a unique (in law) continuous-time stochastic process $X^\infty : [0, \infty) \to \mathcal{VG} \cup \{\infty\}$ with $X^\infty_0 = z$, called the \textbf{continuous time random walk on $\mathcal{G}$ reflected off of infinity}, which satisfies the following properties.
\begin{enumerate}
    \item[(i)] (Almost everywhere defined). For each $t \geq 0$, a.s. $X^\infty_t \in \mathcal{VG}$ and there exists $\varepsilon > 0$ such that $X^\infty_s = X^\infty_t$ for each $s \in (t - \varepsilon, t + \varepsilon)$.
    
    \item[(ii)] (Right continuity). Almost surely, for every $t \geq 0$ such that $X^\infty_t \in \mathcal{VG}$, there exists $\varepsilon > 0$ such that $X^\infty_s = X^\infty_t$ for each $s \in [t, t + \varepsilon]$.
    
    \item[(iii)] (Continuous time random walk). If we let $\tau := \min\{t > 0 : X^\infty_t \neq z\}$, then $\tau$ and $X^\infty_\tau$ are independent, $\tau$ has the exponential distribution with rate $ \mathfrak{m}(z)$, and $X^\infty_\tau$ has the law of a step of the random walk on $\mathcal{G}$ started at $z$ (i.e., for each $y \sim z$, $\mathbb{P}_z^\infty[X^\infty_\tau = y] = \mathfrak{c}(z, y)/\pi(z)$).
    
    \item[(iv)] (Markov property). For every $t \geq 0$ and $x \in \mathcal{VG}$, on the event $\{X^\infty_t = x\}$, the conditional law of $\{X^\infty_{s+t}\}_{s \geq 0}$ given $\{X^\infty_s\}_{s \leq t}$ is the same as the law of $\{X^\infty_s\}_{s \geq 0}$ started from $X^\infty_0 = x$.
    
    \item[(v)] (Recurrence). Almost surely there exist arbitrarily large values of $t \geq 0$ for which $X^\infty_t = z$.
    
    \item[(vi)] (Relation to harmonic functions). Let $A \subset \mathcal{VG}$ be a non-empty finite set and let $\phi : A \to \mathbb{R}$. Let $h^\phi : \mathcal{VG}\to \mathbb{R}$ be the energy-minimizing discrete harmonic function as in Proposition 1.2. If we let $\tau := \min\{t \geq 0 : X^\infty_t \in A\}$, then for each choice of starting vertex $z \in \mathcal{VG}$,
    \[
        h^\phi(z) = \mathbb{E}_z[\phi(X^\infty_\tau)].
    \]
\end{enumerate}
\end{thm}
\begin{rem}\label{remark m_* m_**}
    Recall the construction of the process version of RI involves $\mathfrak{m_{**}}$ (see Definition \ref{Z_t def}), while the construction of the RW$^\infty$ involves $\mathfrak{m_{*}}$ (see Theorem \ref{RW^infty props}). For $(\mathcal{G},\mathfrak{c})$, we first assume that $\mathfrak{m}(x)>\max\{\mathfrak{m}_*(x),\mathfrak{m}_{**}(x)\}$ for all $x\in\mathcal{VG}$. If \ref{EQ} holds for $\mathfrak{m}$, the discrete chains of both models agree in law. Thus, if one model is defined under some $\mathfrak{m}'$ (which may not satisfy the above assumption), so is the other, and \ref{EQ} still holds under $\mathfrak{m}'$. Therefore, \ref{EQ} is essentially a universal property of $(\mathcal{G},\mathfrak{c})$.
\end{rem}

As a continuous-time stochastic process, it is not surprising that RW$^\infty$ satisfies $(i)(ii)(iii)(iv)$. What makes it special is the constant participation of $\infty$ in the process. Transience of random walk on $\mathcal{G}$ enforces RW$^\infty$ to go to $\infty$. Meanwhile, the condition for $\mathfrak{m}$ enforces RW$^\infty$ to jump faster as it ``approaches $\infty$''. Therefore, it can travel to and come back from $\infty$ in finite time, leading to recurrence as in $(v)$. As for $(vi)$, we cite Remark 1.3 in \cite{gs-reflected-rw} to explain why we consider energy-minimizing discrete harmonic functions when studying RW$^\infty$.

Having defined the process $\{X^\infty_t\}$, we can now give some notations. 
For $x\in\mathcal{VG}$, we use $\mathbb{P}^\infty_x$ for the probability measure under which $\{X^\infty_t\}$ starts from $X^\infty_0=x$. 
For any finite $K\subset\mathcal{VG}$, define $\tau_K$ as the first time of hitting $K$. We slightly abuse the notation here by using $\tau_K$ for both RW$^\infty$ and discrete random walk on $\mathcal{G}$. Also, we define $\tau_\infty$ as the first time of $\{X^\infty_t\}$ hits $\infty$. 

\begin{lem}[the strong Markov property of RW$^\infty$ in $\mathcal{VG}$; Lemma 3.10 in \cite{gs-reflected-rw}]\label{VG SM}
    Let $\tau$ be a stopping time for $\{X_t\}_{t \geq 0}$ and let $x \in \mathcal{VG}$. For $z \in \mathcal{VG}$, on the event $\{\tau < \infty, X_\tau = x\}$, the $P_z$-conditional law of $\{X_{s+\tau}\}_{s \geq 0}$ given $\{X_s\}_{s \leq \tau}$ is the same as the $P_x$-law of $\{X_s\}_{s \geq 0}$.
\end{lem}

\begin{lem}[Lemma 3.14 in \cite{gs-reflected-rw}]\label{2.3}
Assume that $\mathcal{G}$ is locally finite. Almost surely, the following is true. Let $b > a > 0$ such that $X^\infty_t \neq \infty$ for each $t \in [a,b]$. Then $\{X^\infty_t\}$ jumps to a different vertex only finitely many times during the time interval $[a, b]$.
\end{lem}

Here we also give a result which is similar to Corollary \ref{hat tau infty a.s.}.
\begin{cor}\label{tau_infty a.s.}
    For any $z\in\mathcal{VG}$, $\mathbb{P}^\infty_z$-a.s. there exist arbitrarily large values of $t \geq 0$ for which $X^\infty_t = \infty$.
\end{cor}
\begin{proof}
    Assume the contrary. Then there exists $a>0$ such that $X^\infty_t \neq \infty$ on $[a,\infty)$ with positive probability. 
    When this event does happen, thanks to Lemma \ref{2.3}, $\{X^\infty_t\}$ jumps to a different vertex only countably many times during the time interval $[a, \infty)$. By Theorem \ref{RW^infty props} (v), $z$ appears in these countably many vertices for infinitely many times. Therefore,
    \[
        \mathbb{P}^\infty_z[X^\infty_t \neq \infty,\forall t\in[a,\infty)]
        \le \sup_{x\in\mathcal{VG}}\mathbb{P}^{\mathrm{RW}}_x[\text{visits $z$ for infinitely many times}]
        =0.
    \]
    Here we use random walk on $(\mathcal{G},\mathfrak{c})$ is transient to deduce the last equality. Therefore, 
    \[
        \mathbb{P}^\infty_z[X^\infty_t \neq \infty,\forall t\in[a,\infty)]
        =0.
    \]
    A contradiction.
\end{proof}

\section{Proof of \ref{EQ}\texorpdfstring{$\Rightarrow$}{⇒}\ref{F = W}}
This section proves \ref{EQ}$\Rightarrow$\ref{F = W}  via two variants of Aldous–Broder algorithm. If \ref{EQ} holds, the 
Aldous-Broder spanning forest generated by $Z|_{[0,\infty)}$ (which is $\mathfrak{c}$-WSF) and the Aldous-Broder spanning forest generated by $X^\infty|_{[\tau_\infty,\infty)}$ (which is $\mathfrak{c}$-FSF) should yeild the same thing, which gives \ref{F = W}. 

We first give definitions of the spanning forests generated by distinct processes, and then present the proof of $\ref{EQ}\Rightarrow\ref{F = W}$. 

Let $\{\mathcal{G}_n\}_{n\ge 1}$ be an increasing family of finite, connected subgraphs such that $\mathcal{G}=\cup_{n \geq 1}\mathcal{G}_n$. Obtain weighted graph $(\mathcal{G}_n^\mathrm{W},\mathfrak{c}_n^\mathrm{W})$ by identifying all the vertices of $\mathcal{G}\setminus\mathcal{G}_n$ to a single vertex $z_n$. 

\begin{defn}[$\mathfrak{c}$-FSF and $\mathfrak{c}$-WSF]\label{FSF WSF def}
    For any finite graph, its $\mathfrak{c}$-spanning tree is defined as a random variable $\mathcal{T}$, such that for each spanning tree $\mathfrak{t}$ of this graph,  
    \[
        \mathbb{P}[\mathcal{T}=\mathfrak{t}]=\frac{1}{Z}\Pi_{e\in\mathfrak{t}}\ \mathfrak{c}(e), 
    \] where $Z$ is a normalizing constant.
    Let $\mathcal{T}^\mathrm{F}_n$ and $\mathcal{T}^\mathrm{W}_n$ be the $\mathfrak{c}$-spanning tree of $\mathcal{G}_n$ and $\mathcal{G}_n^\mathrm{W}$, respectively. 
    Then $\mathfrak{c}$-FSF and $\mathfrak{c}$-WSF are defined as the weak limits of $\mathcal{T}^\mathrm{F}_n$ and $\mathcal{T}^\mathrm{W}_n$, respectively. That is to say, there exists two random spanning forests of $\mathcal{G}$, $\mathcal{T}^{\mathrm{FSF}}$ and $\mathcal{T}^{\mathrm{WSF}}$, such that for any finite set $E\subset\mathcal{EG}$,
    \[
        \mathcal{T}^\mathrm{F}_n\cap E\to\mathcal{T}^\mathrm{FSF}\cap E\quad
        \text{and}\quad\mathcal{T}^\mathrm{W}_n\cap E\to\mathcal{T}^\mathrm{WSF}\cap E
    \]in the total variation sense. For existence and uniqueness, see Section 10.1 of \cite{lyons-peres}. 
\end{defn}

\begin{defn}
    For any $z\in\mathcal{VG}$, consider RW$^\infty$ started from $z$. 
    Take $\eta_x=\inf\{t\mid X^\infty_t=x\}$. Let $p(x)$ be the vertex visited by $\{X^\infty_t\}$ immediately before time $\eta_x$. \textbf{The Aldous-Broder spanning forest of $\mathcal{G}$ generated by $X^\infty|_{[0,\infty)}$} is defined as 
    \[
        \mathcal{T}^\infty(z)=(\mathcal{VT}^\infty(z),\mathcal{ET}^\infty(z)) \text{, where } \mathcal{VT}^\infty(z)=\mathcal{VG},\mathcal{ET}^\infty(z)=\{(p(x),x)\mid x\in \mathcal{VG}\setminus\{z\}\}.
    \]
    For existence of $p(x)$, we refer to Lemma 3.13 in \cite{gs-reflected-rw}. For more about this spanning forest, see Definition 1.13 in \cite{gs-reflected-rw}.
\end{defn}

\begin{defn}
    For any $z\in\mathcal{VG}$, consider RW$^\infty$ started from $z$. 
    Take $\widehat{\eta}_x=\inf\{t>\tau_\infty\mid X^\infty_t=x\}$. Let $\widehat{p}(x)$ be the vertex visited by $\{X^\infty_t\}$ immediately before time $\widehat{\eta}_x$. \textbf{The Aldous-Broder spanning forest of $\mathcal{G}$ generated by $X^\infty|_{(\tau_\infty,\infty)}$} is defined as 
    \[
        \widehat{\mathcal{T}}^\infty(z)=(\widehat{\mathcal{VT}}^\infty(z),\widehat{\mathcal{ET}}^\infty(z)) \text{, where } \widehat{\mathcal{VT}}^\infty(z)=\mathcal{VG},\widehat{\mathcal{ET}}^\infty(z)=\{(\widehat{p}(x),x)\mid x\in \mathcal{VG}\}.
    \]
    For existence of $\hat{p}(x)$, we refer to Lemma 3.13 in \cite{gs-reflected-rw}. 
\end{defn}

\begin{lem}\label{FSF}
     For any $z\in\mathcal{VG}$, $\widehat{\mathcal{T}}^\infty(z)\overset{d}{=}\mathcal{T}^{\mathrm{FSF}}$.
\end{lem}
\begin{proof}
    For any finite edge set $E\subset\mathcal{EG}$, assume that $E\subset\mathcal{EG}_m$ and all neighbors of vertices in $\mathcal{VG}_m$ are in $\mathcal{VG}_M$. 
    Let $\eta_k=\inf\{t>\tau_\infty\mid X^\infty_t\in \mathcal{VG}_k\}$ for all $k$, then $\eta_M<\eta_m$. By Lemma \ref{VG SM}, we have the strong Markov property of RW$^\infty$ for $\eta_M$. Hence, conditioned on $X^\infty_{\eta_M}=x$, 
    the conditional law of $X^\infty|_{[\eta_M,\infty)}$ is the same with the law of $\{X^\infty_t\}$ started from $x$,
    so the conditional law of $E\cap \widehat{\mathcal{T}}^\infty(z)$ is the same as the law of $E\cap\mathcal{T}^\infty(x)$. Therefore, 
    \begin{align*}
        \mathbb{P}[E\cap\widehat{\mathcal{T}}^\infty(z)\in \ \cdot\ ]
        =&\sum_{x\in \mathcal{VG}_M}  \mathbb{P}^\infty_z[X^\infty_{\eta_M}=x]\mathbb{P}[E\cap\widehat{\mathcal{T}}^\infty(z)\in \ \cdot\   \mid X^\infty_{\eta_M}=x]\\
        =&\sum_{x\in \mathcal{VG}_M}  \mathbb{P}^\infty_z[X^\infty_{\eta_M}=x]\mathbb{P}[E\cap\mathcal{T}^\infty(x)\in \ \cdot\ ]\\
        =&\sum_{x\in \mathcal{VG}_M}  \mathbb{P}^\infty_z[X^\infty_{\eta_M}=x]\mathbb{P}[E\cap\mathcal{T}^{\mathrm{FSF}}\in \ \cdot\ ]
        =\mathbb{P}[E\cap\mathcal{T}^{\mathrm{FSF}}\in \ \cdot\ ].
    \end{align*}
    Here the third equality holds because for any $z\in\mathcal{VG}$, $\mathbb{P}^\infty(x)$-a.s. $\mathcal{T}^\infty(x)$ is the $\mathfrak{c}$-FSF of $\mathcal{G}$ (see proof of Theorem 1.14 in \cite{gs-reflected-rw}). 
    By definition and uniqueness of $\mathfrak{c}$-FSF, $\widehat{\mathcal{T}}^\infty(z)\overset{d}{=}\mathcal{T}^{\mathrm{FSF}}$. 
\end{proof}

\begin{defn}
    For the process version of RI, take $\eta_x=\inf\{t\mid Z_t=x\}$, which is finite a.s. because of Corollary \ref{hat tau infty a.s.}. Let $p(x)$ be the vertex visited by $Z_t$ immediately before time $\eta_x$. This is well defined because assume $\eta_x=T(m,a)$, then $p(x)=\omega_m(a-1)$. The Aldous-Broder spanning forest of $\mathcal{G}$ generated by $Z|_{(0,\infty)}$ is defined as 
    \[
        \mathcal{T}^{\mathrm{RI}}=(\mathcal{VT}^{\mathrm{RI}},\mathcal{ET}^{\mathrm{RI}}) \text{, where } \mathcal{VT}^{\mathrm{RI}}=\mathcal{VG},\mathcal{ET}^{\mathrm{RI}}=\{(p(x),x)\mid x\in \mathcal{VG}\}.
    \]
\end{defn}

\begin{proof}[Proof of \ref{EQ}$\Rightarrow$\ref{F = W} in Theorem \ref{main}]
    Let the starting point of RW$^\infty$ be any $x\in\mathcal{VG}$. On the one hand, by Lemma \ref{FSF}, $\widehat{\mathcal{T}}^\infty(x)$ is the $\mathfrak{c}$-FSF of $\mathcal{G}$ $\mathbb{P}^\infty_x$-a.s. 
    On the other hand, according to Theorem 1.1 in \cite{hutchcroft2018interlacements}, $\mathcal{T}^{\mathrm{RI}}$ is the $\mathfrak{c}$-WSF of $\mathcal{G}$ $\mathbb{P}^\mathrm{RI}$-a.s. 

    Since $\widehat{\mathcal{T}^\infty}(x)$ and $\mathcal{T}^{\mathrm{RI}}$ are obtained by running the same algorithm on either $\{X^\infty_t\}_{t\ge\tau_\infty}$ or $\{Z_t\}_{t\ge0}$, \ref{EQ} implies that the $\mathfrak{c}$-FSF of $\mathcal{G}$ has the same law as the $\mathfrak{c}$-WSF of $\mathcal{G}$. This gives \ref{HD = R}. 
\end{proof}

\section{Proof of \ref{HD = R}\texorpdfstring{$\Rightarrow$}{⇒}\ref{EQ} and the strong Markov property}
In Section 4.1 we give definitions and cite previous work on $\ref{F = W}\Rightarrow\ref{HD = R}$. 
In Section 4.2, we employ an electrical network argument to prove Proposition \ref{HD to lim}, which establishes that the two models share identical hitting measures on any finite set $K\subset\mathcal{VG}$. This, combined with the fact that—given an excursion starting at $y$—its remainder follows the law of a continuous-time random walk from $y$, implies a form of excursion similarity when excursions are truncated at their first hitting time of $K$. In Section 4.3, we prove Markov-type properties to establish independence between these excursions (Lemma \ref{Z SM} and Lemma \ref{strong Markov}). In Section 4.4, we synthesize this independence with the excursion similarity to demonstrate the equivalence of the truncated models (Lemma \ref{X^n=Z^n}), from which \ref{EQ} is obtained by passing to the limit. Finally in Section 4.5 we prove Corollary \ref{SM}.

\begin{prop} \label{HD to lim}
    For any finite $K\subset\mathcal{VG}$ and $y\in K$, if \ref{HD = R} holds, then
    \begin{equation}\label{hy&cap Xn}
        \lim_{n\to\infty }h^y_K(X_n)=\mathrm{\tilde{e}}_K(y), \quad\forall X_0=w\in\mathcal{VG},
    \end{equation}
    where $h^y_K = h^{\mathbf{1}_y}_K$ is the energy-minimizing harmonic function as defined in Proposition \ref{h^y def} and $\mathrm{\tilde{e}}_K(y)$ is the normalized equilibrium measure as defined in Definition \ref{capacity def}.
\end{prop}

\subsection{The Dirichlet functional space and electrical currents}

We start by defining the Dirichlet inner product on the space of all functions that map $\mathcal{VG}$ to $\mathbb{R}$: choose a base point $o\in\mathcal{VG}$, then
\[
    \langle f,g\rangle=f(o)g(o)  +\sum_{(x,y)\in \mathcal{EG}} \mathfrak{c}(x,y)[f(x)-f(y)][g(x)-g(y)], 
\]
provided the sum converges absolutely. Note that $\mathcal{EG}$ consists of unoriented edges. 
In particular, let 
\[
    \text{Energy}(f)=\langle f,f\rangle-f(o)^2=\sum_{(x,y)\in \mathcal{EG}} \mathfrak{c}(x,y)[f(x)-f(y)]^2.
\]
Then we can define a Hilbert space of Dirichlet functions
\[
    \textbf{D}(\mathcal{G}) := \{f : \mathcal{VG} \to \mathbb{R} : \text{ Energy}(f) < \infty\}. 
\]
Let $\textbf{D}_0(\mathcal{G})$ be the closure in $\textbf{D}(\mathcal{G})$ of all finitely supported elements of $\textbf{D}(\mathcal{G})$ and $\textbf{HD}(\mathcal{G})$ be all discrete harmonic elements of $\textbf{D}(\mathcal{G})$.

Let $\mathcal{E}$ be the oriented edge set of $\mathcal{G}$. For antisymmetric functions $\theta_1$ and $\theta_2$ on $\mathcal{E}$, we define an inner product
\[
    \langle \theta_1,\theta_2\rangle_{\mathrm{anti}}= \sum_{(x,y)\in \mathcal{EG}}  [\mathfrak{c}(x,y)]^{-1}\theta_1(x,y)\theta_2(x,y)
\]
provided the sum converges absolutely. Let $\mathrm{Energy}(\theta):=\langle \theta,\theta\rangle_{\mathrm{anti}}$. 
Define the flow space as:
\[
    l^2(\mathcal{E},\mathfrak{c})= \{\theta: \mathcal{E}\to\mathbb{R}\mid \forall (x,y)\in\mathcal{E}, \theta(x,y)=-\theta(y,x), \text{ and Energy}(\theta)<\infty\}.
\]
Define the gradient as $\nabla(f)(x,y):=\mathfrak{c}(x,y)[f(x)-f(y)]$. Then $\mathrm{Energy}(f)=\mathrm{Energy}(\nabla f)$. 

For a finite weighted graph, consider disjoint subsets $A$ and $B$ of the vertex set, the voltage $v$ from $A$ to $B$ with boundary condition $v|_{A}\equiv v_1$ and $v|_{B}\equiv v_2$ is a function that is harmonic on $(A\sqcup B)^c$, where $v_1>v_2$. The corresponding current $\theta$ is defined by \textbf{Ohm's Law} $\theta(x,y)=\mathfrak{c}(x,y)[v(x)-v(y)]$. 
Let $\{\mathcal{G}_n\}_{n\ge 1}$ be an increasing family of finite, connected subgraphs such that $\mathcal{G}=\cup_{n \geq 1}\mathcal{G}_n$. Let the corresponding weighted graphs be $\{(\mathcal{G}_n,\mathfrak{c}_n)\}_n$. Identify the vertices outside $\mathcal{G}_n$ to $z_n$, forming the weighted graph $(\mathcal{G}_n^\mathrm{W},\mathfrak{c}_n^\mathrm{W})$. 

\begin{prop}[free current; Proposition 9.1 in \cite{lyons-peres}]\label{free def}
    Let $(a,b)$ be an edge in $\mathcal{G}$, and $\theta_n$ be the unit current flow in $(\mathcal{G}_n,\mathfrak{c}_n)$ from $a$ to $b$. Then there exists $\theta^{(a,b)}_\mathrm{F}\in l^2(\mathcal{E},\mathfrak{c})$, such that $\mathrm{Energy}(\theta_n - \theta^{(a,b)}_\mathrm{F}) \to 0$ as $n \to \infty$ and $\mathrm{Energy}(\theta^{(a,b)}_\mathrm{F}) = \theta^{(a,b)}_\mathrm{F}(a,b) \mathfrak{c}^{-1}(a,b)$. 
\end{prop}
\begin{prop}[wired current; Proposition 9.2 in \cite{lyons-peres}]\label{wired def}
Let $(a,b)$ be an edge in $\mathcal{G}_1$ and $\theta_n$ be the unit current flow in $(\mathcal{G}_n^\mathrm{W},\mathfrak{c}_n^\mathrm{W})$ from $a$ to $b$. There exists $\theta^{(a,b)}_\mathrm{W}\in l^2(\mathcal{E},\mathfrak{c})$, such that $\mathrm{Energy}(\theta_n - \theta^{(a,b)}_\mathrm{W})\to 0 $ as $n \to \infty$ and 
$\mathrm{Energy}(\theta^{(a,b)}_\mathrm{W}) = \theta^{(a,b)}_\mathrm{W}(a,b)\mathfrak{c}^{-1}(a,b)$, which is the minimum energy among all unit flows from $a$ to $b$.
\end{prop}

\begin{defn} \label{free/wired currents def}
    Since $\mathrm{Energy}(\theta)=0\Leftrightarrow \theta$ is the zero function, $\theta^{(a,b)}_\mathrm{F}$ and $\theta^{(a,b)}_\mathrm{W}$ defined in previous propositions are both unique. For $a,b\in\mathcal{VG}$, let $x_1=a, x_2, x_3, ..., x_n=b$ be any path from $a$ to $b$. Then the \textbf{unit free current from $a$ to $b$} and the \textbf{unit wired current from $a$ to $b$} are defined as
    \[
        \sum_{k=1}^{n-1}\theta^{(x_k,x_{k+1})}_\mathrm{F}\quad\text{and}\quad
        \sum_{k=1}^{n-1}\theta^{(x_k,x_{k+1})}_\mathrm{W}\quad,\ \text{respectively.}
    \]
\end{defn}
 
\begin{thm}[Proposition 10.14 in \cite{lyons-peres}]\label{equis for F=W}
    For any $(\mathcal{G},\mathfrak{c})$, let $\mathbb{R}$ denote constant functions, then
    \begin{align*}
        \text{\ref{F = W} $\Leftrightarrow$ free and wired currents are the same, i.e., currents are unique $\Leftrightarrow$ \ref{HD = R}.}
    \end{align*}
\end{thm}

\begin{defn}[Proposition 2.12 in \cite{lyons-peres}]\label{wired infty current}
    Let $\theta_n$ be the unit current in $(\mathcal{G}_n^\mathrm{W},\mathfrak{c}_n^\mathrm{W})$ from $a$ to $z_n$. The pointwise limit $\theta$ of $\{\theta_n\}$ on $(\mathcal{G}, \mathfrak{c})$ is called the \textbf{unit wired current from $a$ to $\infty$}, which is the unique unit flow on $\mathcal{G}$ from $a$ to $\infty$ of minimum energy. Let $v_n$ be the voltage on $(\mathcal{G}_n^\mathrm{W},\mathfrak{c}_n^\mathrm{W})$ corresponding to $\theta_n$ and with $v_n(z_n) := 0$. Then $v := \lim v_n$ exists on $\mathcal{G}$. Moreover, $v$ satisfies $\nabla v=\theta$, $v(a) = \mathrm{Energy}(v) = \mathcal{R}[a \leftrightarrow \infty]$, and $v(x)/v(a) = P_x[\tau_a < \infty]$ for all $x$.
\end{defn}

For any current $\theta_0$ of a finite weighted graph, let its corresponding voltage be $u_0$, then $\theta_0$ satisfies the \textbf{Kirchhoff's Cycle Law}: for any cycle $x_1,x_2,....,x_{n+1}=x_1$,
\[
    \sum_{k=1}^{n} \theta_0(x_k,x_{k+1})\mathfrak{c}^{-1}(x_k,x_{k+1})=\sum_{k=1}^{n} [u_0(x_k)-u_0(x_{k+1})]=0. 
\]
This property is preserved when taking a limit in either energy or pointwise sense. Therefore, fix $o\in\mathcal{VG}$, for any $\theta\in\mathbf{\Theta}$, the following function 
\[
    u(o):=0, u(a):=\sum_{k=1}^{n}\theta(x_k,x_{k+1})\mathfrak{c}^{-1}(x_k,x_{k+1}), \text{ where }x_1=a,x_2,...,x_{n+1}=o\text{ is any path from $a$ to $o$, }
\]
is well-defined. Moreover, $u+\mathbb{R}=\{u:\mathcal{VG}\to\mathbb{R}\mid \nabla u=\theta\}$. 

We now extend the previous definitions of currents from vertices $a, b$ to disjoint finite sets $A, B \subset \mathcal{VG}$, where $B$ could be $\{\infty\}$ for wired case. Consider the free/wired current $\theta$ from $a$ to $b$ after identifying vertices in $A$ to $a$ and $B$ to $b$. By Kirchhoff's Cycle Law, $\{u:\mathcal{VG}\to\mathbb{R}\mid \nabla u=\theta\}$ is a family of functions that differ by an arbitrary constant. Let $u'$ be an element in it, then $\Delta u' (a)\ne 0$. Define $u=u'(a)$ on $A$, $u=u'(b)$ on $B$, and $u=u'$ on $(A\sqcup B)^c$. The \textbf{unit free/wired current from $A$ to $B$} is defined as $[\Delta u' (a)]^{-1}\nabla u$, which still satisfies the Kirchhoff's Cycle Law. 

\begin{defn}
    The current space $\mathbf{\Theta}$ is the vector space spanned by all currents of the following form: the unit free/wired current from $A$ to $B$, and the unit wired current from $A$ to $\infty$, where $A$ and $B$ are arbitrary disjoint finite subsets of $\mathcal{VG}$. 

    By Kirchhoff's Cycle Law, for any $\theta\in\mathbf{\Theta}$, $\{u:\mathcal{VG}\to\mathbb{R}\mid \nabla u=\theta\}$ is still a family of functions that differ by an arbitrary constant. Define it as the set of voltages that corresponds to $\theta$. 
    Then define the voltage space by 
    \[
        \mathbf{U}=\{u:\mathcal{VG}\to\mathbb{R}\mid \nabla u\in \mathbf{\Theta}\}.
    \]
    In particular, the \textbf{free (resp.\ wired) voltage function with boundary values $u_1$ at $A$ and $u_2$ at $B$} ($u_1>u_2$) is defined as the unique function $u$ with the boundary values such that $\nabla u$ is a constant multiple of the unit free (resp.\ wired) current from $A$ to $B$, where $B$ could be $\infty$ for wired condition. We say $u$ has \textbf{boundary condition on $S$} if there exists $(A,B)$ and $u_1>u_2$ such that $A\sqcup B\subset S$ and $u$ is the free/wired voltage function with boundary values $u_1$ at $A$ and $u_2$ at $B$.  
    
    A \textbf{net current} function is a function $i$ on $\mathcal{VG}$, such that for some $\theta\in\mathbf{\Theta}$, $i(x)=\sum_{y\sim x} \theta(x,y)$, or equivalently, $i=\Delta u$ for some $u\in\mathbf{U}$. 
    For $u\in\mathbf{U}$ with boundary conditions on a finite set $S$, we say current $\theta=\nabla u$ and net current $i=\Delta u$ have \textbf{boundary conditions on $S$}. In particular, $i|_{S^c}\equiv 0$, since this property is preserved when taking a limit in either energy or pointwise sense.
\end{defn}

\begin{lem}\label{U in BD}
    Any element in $\mathbf{U}$ is bounded and has finite Dirichlet energy.
\end{lem}
\begin{proof}
    Note that both boundedness and finite Dirichlet energy are preserved under finite linear combinations. Therefore, we only need to prove the following three corner cases. For the voltage from $A$ to $B$, it boils down to the case of $a$ to $b$ in the new graph obtained by identifying $A$ to $a$ and $B$ to $b$. 
    
    \textit{Case 1: $u\in\mathbf{U}$ corresponds to the unit free current from $a$ to $b$. }Let the corresponding net current be $i$. Since the $\Delta u_{(\{a,b\})^c}=i|_{(\{a,b\})^c}\equiv 0$ is preserved when taking a limit in either energy or pointwise sense, by the maximum principle of harmonic functions, 
    \begin{equation}\label{u_A u_B}
        u(b)=\inf_{x\in\mathcal{VG}} u(x), u(a)=\sup_{x\in\mathcal{VG}} u(x), \text{ therefore } \sup_{x,y\in\mathcal{
        VG}}|u(x)-u(y)|=u(a)-u(b). 
    \end{equation}
    Let $x_1=a, x_2,..., x_n=b$ be a path from $a$ to $b$ such that $\{x_k\}_{k=1}^n\cap (\{a,b\})=\{x_1,x_n\}$. Take $u_k\in\mathbf{U}$ such that $u_k+\mathbb{R}=\nabla^{-1}\theta^{(x_k,x_{k+1})}_\mathrm{F}$, then $u+\mathbb{R}=\sum_{k=1}^{n-1} u_k +\mathbb{R}$. Then
    \[
        \mathrm{Energy}(u)=\mathrm{Energy}(\sum_{k=1}^{n-1}u_k)\le (n-1)\sum_{k=1}^{n-1}\mathrm{Energy}(u_k)
        =(n-1)\sum_{k=1}^{n-1}\mathrm{Energy}(\theta^{(x_k,x_{k+1})}_\mathrm{F})<\infty, 
    \]
    \[
        \sup_{x,y\in\mathcal{VG}}|u(x)-u(y)|=u(a)-u(b)\le\sum_{k=1}^{n-1}\sum_{j=1}^{n-1} |u_k(x_j)-u_k(x_{j+1})|
        \overset{\eqref{u_A u_B}}{\le}\sum_{k=1}^{n-1} (n-1) |u_k(x_k)-u_k(x_{k+1})|. 
    \]
    By Proposition \ref{free def}, $|u_k(x_k)-u_k(x_{k+1})|=\mathrm{Energy}(\theta^{(x_k,x_{k+1})}_\mathrm{F})<\infty$, so $\sup_{x,y\in\mathcal{VG}}|u(x)-u(y)|<\infty$.  

    \textit{Case 2: $u\in\mathbf{U}$ corresponds to the unit wired current from $a$ to $b$. }This is analogous to Case 1. 

    \textit{Case 3: $u\in\mathbf{U}$ corresponds to the unit wired current from $a$ to $\infty$. }By Definition \ref{wired infty current}, $u\ge 0$.     By Theorem 2.11 in \cite{lyons-peres}, transience implies that $\mathcal{R}[a \leftrightarrow \infty]$, the effective resistance between $a$ and $\infty$, is finite. Therefore, 
    \[
        0\le u(x)=u(a)P_x[\tau_a < \infty]\le \mathcal{R}[a \leftrightarrow \infty]<\infty, \quad \mathrm{Energy}(u)=\mathcal{R}[a \leftrightarrow \infty]<\infty.
    \]
    
    This gives the proof.
\end{proof}

\subsection{Convergence of harmonic functions}
In this section we will prove Proposition \ref{HD to lim}. Consider (wired) voltages and net currents on the infinite $(\mathcal{G},\mathfrak{c})$ with boundary conditions only on $K\cup\{\infty\}$. We first show that under \ref{HD = R}, all such voltages can be represented by an electrical network with $|K|+1$ points, i.e., the voltage normalized by ``voltage at $\infty$'' can be obtained via a symmetric linear transformation of the net current (see Lemma \ref{G symm}). By Lemma \ref{U in BD}, we interpret $\lim_{n\to\infty} u(X_n)$ as ``voltage at $\infty$'' (guaranteed by Lemma \ref{lim=cons}). 
Once this framework is established, $\mathrm{e}_K(y)$ corresponds to the (wired) net current flowing out from $y$ under a certain boundary condition, while $h^y_K$ corresponds to the (free) voltage function under another boundary condition. \ref{HD = R} and Theorem \ref{equis for F=W} imply that free and wired currents coincide, eliminating the need to distinguish between them. Using the linear relation between net current and normalized voltage, we then derive equation \eqref{hy&cap Xn}.

\begin{lem}\label{lim=cons}
    For $(\mathcal{G},\mathfrak{c})$, if $\textbf{HD}(\mathcal{G})=\mathbb{R}$, for any $f\in\mathbf{D}$, $\lim_{n\to \infty} f(X_n)$ is a constant a.s. This constant is independent of the choice of $X_0$. 
\end{lem}
\begin{proof}
    By transience and Exercise 9.6(f) from \cite{lyons-peres}, we can do Royden Decomposition to $f\in\mathbf{D}$, and uniquely decompose it into the sum of $f_{\mathbf{D}_0}\in\mathbf{D}_0$ and $f_{\mathbf{HD}}\in \mathbf{HD}$. By Theorem 9.11 from \cite{lyons-peres}, $\lim_{n\to \infty} f(X_n)=\lim_{n\to \infty} f_{\mathbf{HD}}(X_n)$. Since $\mathbf{HD}(\mathcal{G})=\mathbb{R}$, we have $f_{\mathbf{HD}}(x)=c$ for all $x$, where the constant $c$ is independent of the choice of $X_0$. Then $\lim_{n\to \infty} f(X_n)=c$ a.s.
\end{proof}

By Lemma \ref{U in BD} and Lemma \ref{lim=cons}, we can define $\alpha: \mathbf{U}\to \mathbb{R}$ such that $\alpha(u)=\lim_{n\to\infty} u(X_n)$. Then $\alpha$ is linear. Define \textbf{the voltage space normalized by ``voltage at $\infty$''} as the vector space
\[
    \textbf{V}=\{v:\mathcal{VG}\to\mathbb{R}\mid \alpha(v)=0\}\subset\mathbf{U}. 
\]
Let $\mathbf{0}$ denote the zero function. Since $\Delta v=\mathbf{0}$ implies $v$ is in $\mathbf{HD}(\mathcal{G})\cap\mathbf{V}$, $\mathbf{HD}(\mathcal{G})=\mathbb{R}$ implies $\Delta$ is injective on $\mathbf{V}$. 

Fix a finite $K\subset\mathcal{VG}$. For any $a\in K$, consider the voltage function $u$ corresponding to unit (wired) current from vertex $a$ to $\infty$, then $\Delta(u)=\mathbf{1}_a$. In this light, the (wired) net current space with boundary conditions only on $K\cup\{\infty\}$ is defined as $\mathbf{I}_K=\{i:\mathcal{VG}\to\mathbb{R}\mid i(x)=0,\forall x\notin K\}\subset\Delta(\mathbf{U})=\Delta(\mathbf{V})$. Since $\Delta$ is injective on $\mathbf{V}$, we can define a bijection $\Delta^{-1}: \mathbf{I}_K\to\Delta^{-1}(\mathbf{I}_K)\subset\mathbf{V}$.

\begin{lem}\label{G symm}
    Green's function $G(x,y)$ is defined as $\frac{1}{\pi(y)}\sum_{n\ge 0}\mathbb{P}^{\mathrm{RW}}_x[X_n=y]$. Define a map $G_K: \mathbf{I}_K\to \Delta^{-1}(\mathbf{I}_K)$ such that $(G_Kf)(x)=\sum_{y\in K} G(x,y)f(y)$. Then 
    \begin{enumerate}
        \item[(i)] $G$ is symmetric.
        \item[(ii)] $G_K=\Delta^{-1}$, so $G_K\Delta $ is an identity map. 
    \end{enumerate}
\end{lem}
\begin{proof}
    For (i), note that
    \begin{align*}
        G(x,y)&=\sum_{n\ge 0}\frac{1}{\pi(y)}\mathbb{P}^{\mathrm{RW}}_x[X_n=y]\\
        &=\frac{1}{\pi(y)}\mathbf{1}_{x=y}+\sum_{n\ge 1}\sum_{p\text{: path from $x$ to $y$ of length $n$}} \frac{1}{\pi(y)}\frac{\mathfrak{c}(x,p_1)}{\pi(x)}\frac{\mathfrak{c}(p_1,p_2)}{\pi(p_1)}...\frac{\mathfrak{c}(p_{n-1},y)}{\pi(p_{n-1})}\\
        &=\frac{1}{\pi(x)}\mathbf{1}_{x=y}+\sum_{n\ge 1}\sum_{p\text{: path from $x$ to $y$ of length $n$}} \frac{1}{\pi(x)}\frac{\mathfrak{c}(x,p_1)}{\pi(p_1)}\frac{\mathfrak{c}(p_1,p_2)}{\pi(p_2)}...\frac{\mathfrak{c}(p_{n-1},y)}{\pi(y)}\\
        &=\sum_{n\ge 0}\frac{1}{\pi(x)}\mathbb{P}^{\mathrm{RW}}_y[X_n=x]=G(y,x).
    \end{align*}
    For (ii), consider any $f\in\mathbf{I}_K$.
    \begin{align*}
        \Delta(G_Kf)(x)
        =&\sum_{z\sim x}\mathfrak{c}(x,z)((G_Kf)(x)-(G_Kf)(z))\\
        =&\sum_{z\sim x}\mathfrak{c}(x,z)\sum_{y\in K}(G(x,y)-G(z,y))f(y)\\
        =&\sum_{y\in K} f(y)\sum_{z\sim x}\mathfrak{c}(x,z)(G(x,y)-G(z,y)).
    \end{align*}
    Note that $\sum_{z\sim x}\mathfrak{c}(x,z)(G(x,y)-G(z,y))=\mathbf{1}_{x=y}$, therefore $\Delta(G_Kf)(x)=f(x)$. Left-multiplying a $\Delta^{-1}$ to both sides and we get:
    \[
        G_K f=(\Delta^{-1}\Delta) G_K f=\Delta^{-1}(\Delta G_K )f=\Delta^{-1}f\quad \forall f\in\mathbf{I}_K. 
    \]
    This gives the proof.
\end{proof}

\begin{proof}[Proof of Proposition \ref{HD to lim}]
    To clarify the proof, we will specify the type (free/wired) of the electrical network in parentheses. Despite that, by \ref{HD = R} and Theorem \ref{equis for F=W}, the free and wired currents coincide.  
    
    Let $\phi\in\mathbf{U}$ be the (wired) voltage function with boundary values 0 at $\infty$ and 1 at vertices in $K$. Let $H_K=\inf\{n\ge0\mid X_n\in K\}, \widetilde{H}_K=\inf\{n>0\mid X_n\in K\}$. By Definition \ref{wired infty current}, we have
    \[
        \phi(z)=\frac{\phi(z)}{1}=\mathbb{P}^{\mathrm{RW}}_z[H_K<\infty],\quad \phi|_K\equiv 1.
    \]
    By definition, $\alpha(\phi)=0$, so $\phi\in\mathbf{V}$. For any $y\in K$, 
    \begin{align*}
        \mathrm{e}_K(y)
        =\pi(y)\mathbb{P}^{\mathrm{RW}}_y[\widetilde{H}_K=\infty]
        =&\sum_{z\sim y,z\notin K}\mathfrak{c}(y,z)\big(1-\mathbb{P}^{\mathrm{RW}}_z[\widetilde{H}_K<\infty]\big)\\
        =&\sum_{z\sim y, z\notin K}\mathfrak{c}(y,z)\big(1-\mathbb{P}^{\mathrm{RW}}_z[H_K<\infty]\big)\\
        =&\sum_{z\sim y}\mathfrak{c}(y,z)\big(1-\phi(z)\big)
        =\sum_{z\sim y}\mathfrak{c}(y,z)\big(\phi(y)-\phi(z)\big)=\Delta\phi(y).
    \end{align*}
    Assume $K=\{k_1,k_2,...,k_m\}$. Since $\Delta\phi\in \mathbf{I}_K$, by Lemma \ref{G symm} (ii), $\forall t=1,2,...,m$,
    \begin{equation}\label{sum a = cap}
        1=\phi(k_t)=G_K\Delta\phi(k_t)
        =\sum_{j=1}^m G(k_t,k_j)\Delta\phi(k_j)=\sum_{j=1}^m G(k_t,k_j)\mathrm{e}_K(k_j). 
    \end{equation}
    
    On the other hand, by Proposition \ref{wired def}, let $\gamma$ be the voltage function of the (wired) unit current from $k_1$ to $K\setminus\{k_1\}$ with $\gamma|_{K\setminus\{k_1\}}\equiv 0$, then $\gamma(k_1)=\gamma(k_1)-0=\ $Energy$(\gamma)$. The (wired) effective resistance between $k_1$ and $K\setminus\{k_1\}$, $\mathcal{R}_w$, is defined as Energy$(\gamma)$, and the (wired) effective conductance, $\mathcal{C}_w$, is defined as its reciprocal. Then $\psi=\mathcal{C}_w\gamma\in\mathbf{U}$ is the (wired) voltage function with boundary values 0 at $K\setminus\{k_1\}$ and 1 at $k_1$. Hence, 
    \begin{align*}
        \mathrm{Energy}(\psi)=\mathcal{C}_w^2\mathrm{Energy}(\gamma)=\mathcal{C}_w^2\mathcal{R}_w=\mathcal{C}_w,
    \end{align*}
    Denote the (free) effective conductance between $k_1$ and $K\setminus\{k_1\}$ as $\mathcal{C}_f$, then
    \begin{align*}
        \mathcal{C}_w&=\mathcal{C}_f && (\text{Theorem \ref{equis for F=W}})\\
        &=\min\{\mathrm{Energy}(F)| F\in\mathbf{D}, F(k_1)=1, F|_{K\setminus\{k_1\}}\equiv 0\} && (\text{Exercise 9.42 in \cite{lyons-peres}}) \\
        &=\min\{\mathrm{Energy}(F)| F(k_1)=1, F|_{K\setminus\{k_1\}}\equiv 0\}
        =\mathrm{Energy}(h^{k_1}_K).
    \end{align*}
    This gives $\mathrm{Energy}(\psi)=\mathrm{Energy}(h^{k_1}_K)$. By uniqueness of energy-minimizing harmonic functions in Proposition 1.2 from \cite{gs-reflected-rw}, $h^{k_1}_K=\psi$. Since $\Delta(\psi-\alpha(\psi))=\Delta\psi\in\mathbf{I}_K$, by Lemma \ref{G symm} (ii), $\forall j=1,2,...,m$,
    \[
       \mathbf{1}_{j=1}-\alpha(\psi)
       =(\psi-\alpha(\psi))(k_j)
       =G_K\Delta\psi(k_j)
       =\sum_{t=1}^m G(k_j,k_t)\Delta\psi(k_t).
    \]
    Multiply each of the above $m$ equations by the corresponding coefficient $\mathrm{e}_K(k_j)$ and sum them up:
    \begin{align*}
       \mathrm{e}_K(k_1)-\mathrm{cap}(K)\alpha(\psi)
       =&\sum_{j=1}^m \sum_{t=1}^m G(k_j,k_t)\Delta\psi(k_t)\mathrm{e}_K(k_j)\\
       =&\sum_{t=1}^m\Delta \psi(k_t)\sum_{j=1}^m G(k_j,k_t)\mathrm{e}_K(k_j)\\
       =&\sum_{t=1}^m\Delta \psi(k_t)\sum_{j=1}^m G(k_t,k_j)\mathrm{e}_K(k_j)  &&(\text{Lemma \ref{G symm} (i)})\\
       \overset{\eqref{sum a = cap}}{=}&\sum_{t=1}^m\Delta \psi(k_t).
    \end{align*}
    
    Since $\Delta\psi$ corresponds to a current from $k_1$ to $K\setminus\{k_1\}$, by definition of net currents, $\sum_{t=1}^m\Delta \psi(k_t)=0$.
    Therefore,
    \[
        \lim_{n\to\infty} h^{k_1}_K(X_n)=\alpha(h^{k_1}_K)=\alpha(\psi)=\tilde{\mathrm{e}}_K(k_1).
    \]
    Similar arguments can be applied to $k_2,...,k_m$ as well.
\end{proof}

\subsection{Independence of excursions}
Here we prove independence between certain excursions for both models, which will be used in Lemma \ref{X^n=Z^n}. For $\{Z_t\}$ we prove that it inherits the strong Markov property in Lemma \ref{R SM}. For $\{X^\infty_t\}$ we prove that it has a restricted version of the strong Markov property under assumption \ref{HD = R}, therefore $\mathbb{P}^\infty_\infty$ can be defined.

Recall that $\{\mathcal{G}_n\}_{n\ge 1}$ is an increasing family of finite, connected subgraphs such that $\mathcal{G}=\cup_{n \geq 1}\mathcal{G}_n$.
\begin{lem}\label{Z SM}
    Recall that $\{Z_t\}$ given in Definition \ref{Z_t def} is the process version of RI. For $x\in\mathcal{VG}$, let $\mathcal{F}^\mathrm{RI}_t$ be the sigma-algebra of $\{Z_s\}_{s\le t}$, and $\tau$ be any stopping time for this filtration. 
    Let $\{X_k\}_{k\ge 0}$ be a random walk started from $x$ that is independent of $\{Z_t\}$, and $t(k)\sim $Exp$(\mathfrak{m}(X_k))$ be conditionally independent random variables conditioned on $\{X_k\}$. Define $\{Z_t\}$ started from $x$ by
    \[
        Z^x_t:= 
        \begin{cases}
            X_k,&\quad \text{ if } \sum_{i<k}t(i)\le t< \sum_{i\le k}t(i)\\
            Z_{t-\sum_{i}t(i)},&\quad \text{ if } t\ge \sum_{i}t(i).\\
        \end{cases} 
    \]
    Denote $\{Z_t\}$ started from $\infty$ by $\{Z_t\}$ itself. Then on the event $\{\tau<\infty\}$, the conditional law of $Z|_{[\tau,\infty)}$ given $Z|_{[0,\tau]}$ is the same as the law of $\{Z_t\}$ started from $Z_\tau$.
\end{lem}
\begin{proof}
    Denote the corresponding sigma-algebra of $\{Z_t\}_{0\le t\le \tau}$ as  $\mathcal{F}^\mathrm{RI}_\tau$. Recall that RI is a random measure-valued process, with sample points $\boldsymbol{\omega}=\sum_{m\ge 0}\delta_{(w^*_m,u_m)}$. Recall definition of $t(m,n)$ and $T(m,n)$ from Definition \ref{T(m,n) def} and define 
    \[
        T(m)=\lim_{n\to\infty} [T(m,n)+t(m,n)].
    \]
    
    By Definition \ref{Z_t def}, $Z_0=\infty$ a.s. Therefore, $\{Z_t\}$ ``starts from $\infty$''. 
    
    \textit{Case 1: Condition on $\{\tau<\infty, Z_\tau=\infty\}$}. Let $M_1=\{m\mid T(m)\le \tau\}$ and $M_2=\{m\mid T(m)> \tau\}$. Let $\eta=\sup_{m\in M_1} u_m\le \inf_{m\in M_2} u_m<\infty$ and $\hat{R}_t=R_{t+\eta}-R_\eta$. 
    By Lemma \ref{R SM}, since $\eta$ is finite, $\{\hat{R}_t\}_{t\ge 0}\overset{d}{=}\{R_t\}_{t\ge 0}$. 
    
    Let $\hat{Z}_t$ be the process version RI generated by $\hat{R}_t$, then $\{\hat{Z}_t\}_{t\ge0}\overset{d}{=}\{Z_t\}_{t\ge0}$. For all $m\in M_2$, 
    \[
        T(m,n)=\sum_{(k,j)\in I(m,n)}t(k,j)=\sup_{m\in M_1}T(m)+\sum_{(k,j)\in \hat{I}(m,n)}t(k,j)=\tau+\hat{T}(m,n).
    \]
    Note also that $\hat{Z}_0=\infty=Z_\tau$ a.s., so
    \[
        \{\mathbb{E}^\mathrm{RI}[Z_{\tau+t}|\mathcal{F}^\mathrm{RI}_\tau]\}_{t\ge0}
       \overset{d}{=}\{\hat{Z}_t\}_{t\ge0}
       \overset{d}{=}\{Z_t\}_{t\ge0}. 
    \]
    
    \textit{Case 2: Condition on $\{\tau<\infty, Z_\tau=x\in\mathcal{VG}\}$}. Let $\tau_0=\inf\{t>\tau\mid Z_t=\infty\}$ be a stopping time. By Case 1, 
    \[
        \{Z_t\}_{t\ge 0}\overset{d}{=}
        \{\mathbb{E}^\mathrm{RI}[Z_{\tau_0+t}|\mathcal{F}^\mathrm{RI}_{\tau_0}]\}_{t\ge 0}. 
    \]
    Assume $\tau\in[T(m,n),T(m,n)+t(m,n))$, by definition of RI, 
    $\{w_m(n+k)\}_{k\ge 0}$ has the law of random walk started from $x$. Note also that $Z_0=\infty$, therefore
    \[
        \{Z^x_t\}_{0\le t<\sum_i t(i)}\overset{d}{=}
        \{\mathbb{E}^\mathrm{RI}[Z_{\tau+t}|\mathcal{F}^\mathrm{RI}_{\tau}]\}_{0\ge t<\tau_0}, 
    \]
    \[
        Z^x_{\sum_it(i)}\overset{d}{=}Z_0=\infty,\quad \mathbb{E}^\mathrm{RI}[Z_{\tau_0}|\mathcal{F}^\mathrm{RI}_{\tau}]=\infty,
    \]
    \[
        \big\{\ \mathbb{E}^\mathrm{RI}\big[\ Z^x_{t+\sum_i t(i)}\mid \sigma(Z^x_t:t\le\sum_{i} t(i))\ \big]\ \big\}_{t\ge 0}
        \overset{d}{=}\{Z_t\}_{t\ge0}\overset{d}{=}
        \{\mathbb{E}^\mathrm{RI}[Z_{\tau_0+t}|\mathcal{F}^\mathrm{RI}_{\tau_0}]\}_{t\ge0}. 
    \]
    This means 
    \[
        \{\mathbb{E}^\mathrm{RI}[Z_{\tau+t}|\mathcal{F}^\mathrm{RI}_{\tau}]\}_{t>0}
        \overset{d}{=}\{Z^x_t\}_{t>0},
    \]
    which completes the proof. 
\end{proof}

The following lemma establishes a restricted strong Markov property for RW$^\infty$, which holds specifically for stopping times of the form $\tau_\infty$.
\begin{lem}\label{strong Markov}
    Assume \ref{HD = R} holds. Consider RW$^\infty$ started from $z\in\mathcal{VG}$ and let $\tau_\infty=\inf\{t:X^\infty_t=\infty\}$. Then the law of $\left\{X^\infty_s\right\}_{s\ge\tau_\infty}$ does not depend on $\left\{X^\infty_s\right\}_{s\le\tau_\infty}$ and the choice of $z$. We define this conditional law to be $\mathbb{P}^\infty_\infty$. Moreover, let $\{X_n\}$ be random walk on $(\mathcal{G},\mathfrak{c})$ with arbitrary starting point, and then
    \begin{equation}\label{P_inf charac}
        \mathbb{P}_\infty^\infty[X^\infty_{\tau_K}=y]=\lim_{m\to\infty} h^y_K(X_m)\quad a.s.
    \end{equation}
\end{lem}
\begin{proof}
     By Proposition \ref{HD to lim}, for arbitrary starting point of $\{X_m\}$, 
    \[
        \lim_{m\to\infty} h^y_K(X_m)=\tilde{\mathrm{e}}_K(y).
    \]

    Let $T_K =\inf\left\{t>\tau_\infty : X^\infty_t\in K\right\}$. Since $\tau_\infty$ is a stopping time, $T_K$ is also a stopping time. Define the vertices (can be repeated) that $X^\infty|_{[0,\tau_\infty]}$ hit in time order as $\hat{X}_0, \hat{X}_1,...$ By independence in Theorem \ref{RW^infty props} (iii), we have 
    \begin{align*}
        \mathbb{P}^\infty_z[X^\infty_{T_K}=y\mid \{X^\infty_t\}_{t<\tau_\infty}]
        =\mathbb{P}^\infty_z[X^\infty_{T_K}=y\mid \{\hat{X}_m\}_{m\ge 0}].
    \end{align*}     
    Let $\phi(x)=\mathbb{P}^\infty_{x}[X^\infty_{T_K}=y]$. Since it's harmonic, $\{\phi(\hat{X}_M)\}_{ M\in\mathbb{N}_0}$ is a martingale. By Lemma \ref{VG SM}, 
    \[
        \mathbb{E}^\infty_z[\mathbf{1}_{X^\infty_{T_K}=y}\mid \{\hat{X}_m\}_{0\le m\le M}]
        =\mathbb{P}^\infty_z[X^\infty_{T_K}=y\mid \{\hat{X}_m\}_{0\le m\le M}]
        =\phi(\hat{X}_M). 
    \]
    Since $\mathbf{1}_{X^\infty_{T_K}=y}\in L^1$, we know that $\{\phi(\hat{X}_M)\}_{ M\in\mathbb{N}_0}$ is a uniformly integrable martingale, so 
    \begin{align*}
        \mathbb{P}^\infty_z[X^\infty_{T_K}=y\mid \{X^\infty_t\}_{t<\tau_\infty}]
        =\mathbb{P}^\infty_z[X^\infty_{T_K}=y\mid \{\hat{X}_m\}_{m\ge 0}]
        =\lim_{M\to\infty} \phi(\hat{X}_m)\quad a.s.
    \end{align*}
    Since $\{X_m\}_{m\ge 0}$ will eventually leave $K$ with probability 1, 
    \begin{align*}
        \mathbb{P}^\infty_z[X^\infty_{T_K}=y\mid \{X^\infty_t\}_{t<\tau_\infty}]
        =&\lim_{M\to\infty} \mathbb{P}^\infty_{\hat{X}_M}[X^\infty_{T_K}=y]\\
        =&\lim_{M\to\infty} \mathbb{P}^\infty_{\hat{X}_M}[X^\infty_{\tau_K}=y]
        =\lim_{M\to\infty} h^y_K(\hat{X}_M) = \tilde{\mathrm{e}}_K(y),
    \end{align*}
    which is a constant that is independent of $\left\{X^\infty_s\right\}_{s\le\tau_\infty}$. This gives \eqref{P_inf charac}.
    
    By Lemma \ref{VG SM}, $\left\{X^\infty_s\right\}_{s\ge T_K}$ and $\left\{X^\infty_s\right\}_{s\le\tau_\infty}$ are independent. In particular, we can take $K=\mathcal{VG}_n$ for any $n$. Since $T_{\mathcal{VG}_n}$ decreases with $n$, by the definition of $\{X^\infty_t\}$, $\forall t\in [\tau_\infty, \lim_{n\to \infty}T_{\mathcal{VG}_n}]$, $X^\infty_{t}=\infty$. This gives the desired result.
\end{proof}

\subsection{Proof of \ref{HD = R}\texorpdfstring{$\Rightarrow$}{⇒}\ref{EQ}}
This subsection is dedicated to prove \ref{HD = R}$\Rightarrow$\ref{EQ}. In Lemma \ref{X^n=Z^n} we prove the truncated versions of $\{X^\infty_t\}_{t\ge\tau_\infty}$ and $\{Z_t\}_{t\ge 0}$ are equivalent. In Lemma \ref{d(Z,Z^n)} we show for both $\{X^\infty_t\}_{t\ge\tau_\infty}$ and $\{Z_t\}_{t\ge 0}$, the sequence of truncated processes converges to the original process. Combining these results yields \ref{EQ}.

Recall $\{\mathcal{G}_n\}_{n\ge 1}$ is an increasing family of finite, connected subgraphs such that $\mathcal{G}=\cup_{n \geq 1}\mathcal{G}_n$. For either $\{X^\infty_t\}_{t\ge\tau_\infty}$ or $\{Z_t\}_{t\ge 0}$, we restrict our attention to excursions that intersect a finite set $\mathcal{VG}_n$. For any such excursion, we start observing its behavior once it hits $\mathcal{VG}_n$, continue observing until it hits $\infty$, we then stop observing until the arrival of another such excursion. By concatenating what we have observed, we get $X^n$ or $Z^n$.

\begin{lem}\label{X^n=Z^n}
    Fix $n\ge 1$. 
    For any $m\ge 1$, define
    \[
        \sigma^n_0=\tau_\infty,\ 
        \tau_m^n=\inf\left\{t>\sigma_{m-1}^n\mid X^\infty_t\in\mathcal{VG}_n\right\},\
        \sigma_m^n=\inf\left\{t>\tau_{m}^n\mid X^\infty_t=\infty\right\}.
    \]
    \[
        \hat{\sigma}^n_0:=0,\ 
        \hat{\tau}_m^n=\inf\left\{t>\hat{\sigma}_{m-1}^n\mid Z_t\in\mathcal{VG}_n\right\},\
        \hat{\sigma}_m^n=\inf\left\{t>\hat{\tau}_{m}^n\mid Z_t=\infty\right\}.
    \]
    
    For any $t\in [0,\ \infty)$, let $k(t)=\max\left\{k\ge 0\mid \sum_{m=1}^k (\sigma_m^n-\tau_m^n)<t\right\}$. Let
    \[
        X^n_t:=X^\infty_{\tau_{k(t)+1}^n+t-\sum_{m=1}^{k(t)} (\sigma_m^n-\tau_m^n)}.
    \]
    
    Similarly we define $Z^n_t$. Assume that \ref{HD = R} and \eqref{hy&cap Xn} hold, then: 
    \begin{align*}
        \left\{X^n_t\right\}_{t\ge 0} \text{ under } \mathbb{P}^\infty_x \text{ and } \left\{Z^n_t\right\}_{t\ge 0} \text{ under } \mathbb{P}^\mathrm{RI} \text{ are identically distributed}.
    \end{align*}
\end{lem}
\begin{proof}
    On the one hand, thanks to Theorem \ref{RW^infty props} (v) and Corollary \ref{tau_infty a.s.}, $\tau^n_m<\infty, \sigma^n_m<\infty$ a.s. Since \ref{HD = R} holds, by Lemma \ref{strong Markov} and Lemma \ref{VG SM}, $\left\{X^\infty|_{[\tau^n_m,\sigma^n_m)}\right\}_{m\ge 1}$ are mutually independent. By \eqref{P_inf charac} and \eqref{hy&cap Xn},
    \[
        \mathbb{P}^\infty_z[X^\infty_{T_K}=y \mid \{X^\infty_s\}_{s\le\tau_\infty}]=\tilde{\mathrm{e}}_K(y). 
    \]
    Hence, for any $B\in \mathcal{W}_+$ such that all elements in $B$ starts from $y\in \mathcal{VG}_n$, 
    \begin{align*}
        &\mathbb{P}^\infty_\infty[ \text{the order of vertices that }X^\infty|_{[\tau^n_m,\sigma^n_m)} \text{ hits in $\mathcal{VG}$ is exactly the same as that of an element in $B$} ]\\
        =&\tilde{\mathrm{e}}_{\mathcal{VG}_n}(y)\mathbb{P}^{\mathrm{RW}}_y[B].
    \end{align*}

    On the other hand, thanks to Corollary \ref{hat tau infty a.s.}, $\hat{\tau}^n_m<\infty, \hat{\sigma}^n_m<\infty$ a.s. By Lemma \ref{Z SM}, $\left\{Z|_{[\hat{\tau}^n_m,\hat{\sigma}^n_m)}\right\}_{m\ge 1}$ is i.i.d. with distribution
    \begin{align*}
        &\mathbb{P}^\mathrm{RI}[ \text{the order of vertices that } Z|_{[\hat{\tau}^n_m,\hat{\sigma}^n_m)}\text{ hits in $\mathcal{VG}$ is exactly the same as that of an element in $B$} ]\\
        =&\tilde{\mathrm{e}}_{\mathcal{VG}_n}(y)\mathbb{P}^{\mathrm{RW}}_y[B].
    \end{align*}
    Since exponential waiting times are sampled independently of the hitting sequence and governed by the public parameter $\mathfrak{m}$, the lemma is thereby established.
\end{proof}

Now we define a metric between continuous-time processes and prove that truncated versions are close to their original process in the sense of this metric. 
\begin{defn}\label{metric}
    Let $M$ be the set of all Lebesgue measurable functions $f:\ \left(0,\ \infty\right)\rightarrow\mathcal{VG}\cup\left\{\infty\right\}$ modulo re-definition on a set of zero Lebesgue measure. Then we can define a metric on $M$ as
    \[
        d(f,g)=\int_0^{\infty} e^{-t}\mathbf{1}_{f(t)\ne g(t)} dt.
    \]
\end{defn}

\begin{lem}\label{d(Z,Z^n)}
    $\lim_{n\rightarrow\infty}\mathbb{E}^\infty_\infty[d(X^\infty,X^n)]=0,\ \lim_{n\rightarrow\infty}\mathbb{E}^\mathrm{RI}[d(Z,Z^n)]=0$.
\end{lem}
\begin{proof}
    Fix t. Since $\{X^\infty_t\}$ started from $\infty$ can be viewed as $\{X^\infty_t\}$ started from $\mathcal{VG}$ on $(\tau_\infty,\infty)$, we can still apply Theorem \ref{RW^infty props} $(i)$, so $\mathbb{P}^\infty_\infty-$a.s. we can pick $\epsilon>0$ such that $X^\infty_s=X^\infty_t$ for all $s\in(t-\epsilon, t+\epsilon)$. Additionally, we define remaining time $\Delta^n_t$ as the Lebesgue measure of $[0,t]\setminus \cup_{m=1}^\infty[\tau^n_m,\sigma^n_m)$, so $\lim_{n\rightarrow\infty} \Delta^n_t=0$. Therefore,
    \[
        \mathbb{P}^\infty_\infty[X^\infty_t\ne X^n_t]
        \le \mathbb{P}^\infty_\infty[\Delta^n_t\ge \epsilon]\rightarrow 0  \quad(n\rightarrow\infty).
    \]
    By Fubini's theorem and Dominated Convergence theorem,
    \[
        \mathbb{E}^\infty_\infty[d(X^\infty,X^n)]
        =\int_0^{\infty} e^{-t}\mathbb{P}^\infty_\infty[X^\infty_t\ne X^n_t] dt\rightarrow 0 \quad(n\rightarrow\infty).
    \]
    
    Similar proof can be applied to $\{Z_t\}$. It remains to check that for any $t>0$, a.s. we can pick $\epsilon>0$ such that $Z_s=Z_t$ for all $s\in(t-\epsilon, t+\epsilon)$. This is just an analogy to Lemma 3.7 in \cite{gs-reflected-rw}.
    
    Note that for any $(k,j)$, 
    \[
        T(k,j)=\sum_{(m,n)\in I(k,j)} t(m,n)=\sum_{(m,n)\in I(k,j)} \text{Leb}((T(m,n),T(m,n)+t(m,n))).
    \]
    here $I(k,j)$ is as defined in \eqref{I(k,j)}. Since all these open intervals are disjoint, we know that
    \[
        \text{Leb}([0,t]\setminus \cup_{m\ge 1,n\in\mathbb{Z}}(t-T(m,n)-t(m,n),t-T(m,n)))=0\quad a.s.
    \]  
    Let $U\sim \text{Uniform}(0,t)$ be independent from everything else. Then for any $(k,j)$, $\left\{t(m,n)\right\}_{m\le 1, n\in\mathbb{Z}}$ and $\left\{t(m,n)+U\mathbf{1}_{m=j,n=k}\right\}_{m\le 1, n\in\mathbb{Z}}$ are mutually absolutely continuous. Therefore, $\left\{T(m,n)\right\}_{(m,n)\in I(k,j)}$ and $\left\{T(m,n)+U\right\}_{(m,n)\in I(k,j)}$ are mutually absolutely continuous.
    So we have
    \begin{align*}
        &\mathbb{P}[t\in [0,\infty)\setminus \cup_{m\le 1, n\in\mathbb{Z}}(T(m,n)+U,T(m,n)+t(m,n)+U)]\\
        \;\le &\inf_{(k,j)}\mathbb{P}[t\in [0,\infty)\setminus \cup_{(m,n)\in I(k,j)}(T(m,n)+U,T(m,n)+t(m,n)+U)]\\
        \;=&\inf_{(k,j)}\mathbb{P}[U\in [0,t]\setminus \cup_{(m,n)\in I(k,j)}(t-T(m,n)-t(m,n),t-T(m,n))]\\
        \;=& \frac{1}{t}\inf_{(k,j)} \mathbb{E}[\text{Leb}([0,t]\setminus \cup_{(m,n)\in I(k,j)}(t-T(m,n)-t(m,n),t-T(m,n)))]\\
        \;=& \frac{1}{t} \mathbb{E}[\text{Leb}([0,t]\setminus \cup_{m\ge 1,n\in\mathbb{Z}}(t-T(m,n)-t(m,n),t-T(m,n)))]
        =0.
    \end{align*}
    By the absolute continuity, 
    \[
        \mathbb{P}[t\in [0,\infty)\setminus \cup_{m\le 1, n\in\mathbb{Z}}(T(m,n),T(m,n)+t(m,n))]=0.
    \]
    Thus, a.s. we can pick $\epsilon>0$ such that $Z_s=Z_t$ for all $s\in(t-\epsilon, t+\epsilon)$.
\end{proof}

\begin{proof}[Proof of \ref{HD = R}$\Rightarrow$\ref{EQ} in Theorem \ref{main}]
    Let $\mu_X$ and $\mu_Z$ be the laws of $X^\infty|_{[\tau_\infty,\infty)}$ and $Z|_{[0,\infty)}$ on the metric space $(M,d)$ defined in Definition \ref{metric}. Fix $\epsilon > 0$. By Lemma \ref{d(Z,Z^n)}, choose $N$ such that  
    \[
        \mathbb{E}^\infty_\infty[d(X^\infty,X^N)] < \epsilon^2/8 \quad \text{and} \quad \mathbb{E}^\mathrm{RI}[d(Z,Z^N)] < \epsilon^2/8.
    \]  
    Under \ref{HD = R} and \eqref{hy&cap Xn}, Lemma \ref{X^n=Z^n} gives a coupling $\mathbb{P}^N$ with $d(X^N, Z^N) = 0$ a.s. Then by Markov's inequality,  
    \begin{align*}
        \mathbb{P}^N[d(X^\infty,Z) \ge \epsilon] 
        \le \mathbb{P}^\infty[d(X^\infty,X^N) \ge \frac{\epsilon}{2}] +\mathbb{P}^\mathrm{RI}[d(Z,Z^N) \ge \frac{\epsilon}{2}] 
        \le\frac{\mathbb{E}^\infty_\infty[d(X^\infty,X^N)] + \mathbb{E}^\mathrm{RI}[d(Z,Z^N)]}{\epsilon/2} \le \epsilon.
    \end{align*}
    For any Borel set $A \subset M$, let $A_\epsilon =  \{f \in M :\ \exists\ a \in A \text{ such that } d(a,f)<\epsilon\}$. Then  
    \[
        \mu_Z(A_\epsilon) + \epsilon \ge \mu_X(A) \quad \text{and} \quad \mu_X(A_\epsilon) + \epsilon \ge \mu_Z(A),
    \]  
    Since $\epsilon$ is arbitrary, the Prokhorov distance $\pi(\mu_X, \mu_Z) = 0$. Hence $\mu_X = \mu_Z$, which implies \ref{EQ}. 
\end{proof}

\subsection{Proof of Corollary \ref{SM}}

\begin{proof}[Proof of Corollary \ref{SM}]
    Since the graph satisfies conditions in Theorem \ref{main}, $\left\{X_t^\infty\right\}_{t\ge0}$ under $\mathbb{P}^\infty_\infty$ and $\left\{Z_t\right\}_{t\ge0}$ under $\mathbb{P}^\mathrm{RI}$ are identically distributed. By Lemma \ref{Z SM}, $\{Z_t\}$ has the strong Markov property, so $\left\{X_t^\infty\right\}_{t\ge0}$ under $\mathbb{P}^\infty_\infty$ has the strong Markov property. 
    
    Additionally, for RW$^\infty$ started from $z\in\mathcal{VG}$, by Lemma \ref{strong Markov}, $\left\{X_t^\infty\right\}_{t>\tau_\infty}$ is independent of $\left\{X_t^\infty\right\}_{t<\tau_\infty}$ and $z$. By Lemma \ref{VG SM}, RW$^\infty$ started from $z\in\mathcal{VG}$ has the strong Markov property. 
\end{proof}

\section{Examples}
In this section we give examples for condition \ref{HD = R}. By Theorem \ref{main}, for any graph that satisfies \ref{HD = R}, RI and RW$^\infty$ are equivalent and have the strong Markov property. For any that does not, RI and RW$^\infty$ are not equivalent. 

\begin{defn}[Ends]
    For an infinite graph $\mathcal{G}$, we define an equivalence relation on all self-avoiding path on $\mathcal{G}$ as follows:  
    $(x_n)_{n\in\mathbb{N}_0}$ is equivalent to $(y_n)_{n\in\mathbb{N}_0}$ if and only if for any finite $V\subset\mathcal{VG}$, there exists a unique component of $\mathcal{VG}\setminus V$ that contains infinitely many $x_n$ and infinitely many $y_n$. The equivalence classes are called ends of graph $\mathcal{G}$. 
\end{defn}
\begin{rem}\label{end(x)}
In particular, Example \ref{Z^d+N} demonstrates that there exists graphs with multiple ends for which \ref{HD = R} holds, while Example \ref{Zd + Zd} demonstrates that \ref{HD = R} could fail for one-ended graphs. Therefore, the concept of ends is not an appropriate characterization of ``points at infinity'' in our setting. 
\end{rem}

\subsection{\texorpdfstring{$\mathbb{Z}^d$}{Zd} and its variants}

\begin{exmp}[$\mathbb{Z}^d$ with unit conductance]\label{Z^d}
    Consider lattice $\mathbb{Z}^d(d\ge 3)$ with the same conductance on each edge. Then random walk on this graph is transient. By Theorem 9.7 in \cite{lyons-peres}, consider boxes of radius $n$ centered at the origin, we have that currents are unique. By Theorem \ref{equis for F=W}, \ref{HD = R} stands. 
\end{exmp}

\begin{exmp}[a graph with multiple ends for which \ref{HD = R} holds]\label{Z^d+N}
    Consider the graph $\mathcal{G}$ obtained by identifying the origins of $\mathbb{N}_0$ and $\mathbb{Z}^d (d\ge 3)$ and equip the graph with unit conductance. Since the $\mathbb{N}_0$ part is recurrent, this graph is still transient. Consider any $f\in\mathbf{HD}(\mathcal{G})$, harmonicity implies $f$ must be an arithmetic sequence on $\mathbb{N}_0$, therefore finiteness of Dirichlet energy implies $f$ must be constant on $\mathbb{N}_0$. In this light, we can see that in fact $f|_{\mathbb{Z}^d}\in \mathbf{HD}(\mathbb{Z}^d)=\mathbb{R}$, so we get $f$ is the zero function. That is to say, \ref{HD = R} is satisfied. 
\end{exmp}

\begin{exmp}\label{question}
    We consider a ``ladder-like'' variant of the previous example. Consider the graph $\mathcal{G}$ obtained by adding edges and conductances of $c_k$ between $k\in\mathbb{N}_0$ and $(k,0,0,...,0)\in\mathbb{Z}^d$ for any $k=1,2,3,...$ to the graph in the previous example, where the value of $c_k$ is to be determined. For random walk started from a non-origin vertex in $\mathbb{Z}^d$, there is positive probability that it never hit the origin, so this graph is transient. 
    By Corollary 1.2 in \cite{carmesin2012characterization}, we can take $S$ in the corollary to be the set of edges that we just added, and then $\sum_{h=1}^{\infty} c_h<\infty$ satisfies the corollary condition. Since the previous example satisfies \ref{HD = R}, we know that $(\mathcal{G},\mathfrak{c})$ satisfies \ref{HD = R}. 
\end{exmp}

\subsection{Other examples}

\begin{exmp}
    Consider $\mathcal{G}\times A$, where $(\mathcal{G},\mathfrak{c}_1)$ is locally finite, transient, and satisfies \ref{HD = R}, and $A$ is a finite graph equipped with any conductance function $\mathfrak{c}_2$. The edge set is defined as
    \[
        E:=\{[(x,y),(x',y')]\mid x=x' \text{ and } y\sim y' \text{ in } A \text{ or } y=y' \text{ and } x\sim x' \text{ in } \mathcal{G} \},
    \]
    and the conductance function is defined as
    \[
        \mathfrak{c}([(x,y),(x',y')]):=
        \begin{cases}
             \mathfrak{c}_2(y,y'),&\quad \text{if }x=x'\\
             \mathfrak{c}_1(x,x'),&\quad \text{if }y=y'.
        \end{cases}
    \]
    By Exercise 9.45(a) of \cite{lyons-peres}, $(\mathcal{G},\mathfrak{c}_1 )$ satisfies \ref{HD = R} yields $(\mathcal{G}\times A,\mathfrak{c})$ satisfies \ref{HD = R}. 

    As a consequence, if random walk on the Cayley graph of a finitely generated abelian group $H$ is transient, this Cayley graph equipped with unit conductance satisfies \ref{HD = R}. The reason is as follows. By the Fundamental Theorem of Finitely Generated Abelian Groups (see, for example, Section 5.2 of \cite{dummit2004abstract}), $H$ can be represented as $\mathbb{Z}^d\times A$ for some $d\ge 0$ and $A$ finite with unit conductance. Transience and coarse-graining ensure that $d\ge 3$. By Example \ref{Z^d}, we can apply the conclusion of the previous paragraph. 
\end{exmp}

\begin{exmp}
    \ref{HD = R} holds for many Cayley graphs (with unit conductance). 
    
    (i) Construct a Cayley graph with a finite generating set $S$ of an infinite discrete Kazhdan group such that $S^{-1}=S$. Such graphs exists because by Theorem 1.3.1 in \cite{bekka2007kazhdan}), any discrete Kazhdan group is finitely generated. 
    By \cite{bekka1997group}, the Cayley graph of any infinite, finitely generated Kazhdan group satisfies \ref{HD = R}. 
   In particular, \ref{HD = R} holds for $SL_n(\mathbb{Z})(n\ge 3)$. 
    
    (ii) Define $x$ and $y$ as below and consider a Cayley graph of the discrete Heisenberg group $H_3(\mathbb{Z})$ with generating set $\{x,y,x^{-1},y^{-1}\}$. 
    \[
        {\displaystyle x={\begin{pmatrix}1&1&0\\0&1&0\\0&0&1\\\end{pmatrix}},\ \ y={\begin{pmatrix}1&0&0\\0&1&1\\0&0&1\\\end{pmatrix}},\ \ } 
        {\displaystyle z={\begin{pmatrix}1&0&1\\0&1&0\\0&0&1\\\end{pmatrix}}=xyx^{-1}y^{-1}}.
    \]
    Consider the center $\{kz\mid k\in\mathbb{Z}\}\cong\mathbb{Z}$, this is an infinite, normal, infinite-index subgroup that is finitely generated, so \ref{HD = R} is satisfied (see Theorem 6.8 of \cite{Gaboriau2002}). 
\end{exmp}

\begin{defn}
    An automorphism on $(\mathcal{G},\mathfrak{c})$ is a bijection $\phi:\ \mathcal{VG}\to\mathcal{VG}$ such that 
    \[
        \forall x,y\in\mathcal{VG}, p(x,y)=p(\phi(x),\phi(y)). 
    \]
    
    $(\mathcal{G},\mathfrak{c})$ is said to be transitive if $\forall x,y\in \mathcal{VG}$, there exists an automorphism that maps $y$ to $x$. $(\mathcal{G},\mathfrak{c})$ is said to be quasi-transitive if it has only finitely many orbits for the action of the automorphism group on $\mathcal{VG}$. 
\end{defn}

\begin{exmp}
    Consider any amenable transient graph $\mathcal{G}$. 
    If it is any of the following cases, \ref{HD = R} holds: (i) transitive; (ii) quasi-transitive and equipped with unit conductance; (iii) quasi-transitive and has bounded degrees.
    By Theorem 9 in \cite{elek1998quasi}, \ref{HD = R} is satisfied for (i). By Theorem \ref{equis for F=W} and Exercise 10.11 in \cite{lyons-peres}, \ref{HD = R} is satisfied for (ii). Based on (ii), by Theorem 9.9 in \cite{lyons-peres}, \ref{HD = R} is satisfied for (iii). 
\end{exmp}

\subsection{Examples that fail the condition of Theorem \ref{main}}

\begin{exmp}
    All transient trees do not satisfy \ref{EQ} (see Exercise 9.35 in \cite{lyons-peres}). 
    
    More generally, \ref{EQ} fails on any transient planar graph with bounded $\pi(\cdot)$. Recall the stationary measure $\pi(\cdot)$ is defined as $\pi(x)=\sum_{y\sim x}\mathfrak{c}(x,y)$. By Theorem \ref{equis for F=W} and Theorem 9.12 of \cite{lyons-peres}, \ref{HD = R} is not satisfied. Since we have Theorem \ref{main}, \ref{EQ} fails. 
\end{exmp}

\begin{exmp}[\ref{HD = R} could fail for one-ended graphs]\label{Zd + Zd}
    Take $d\ge 3$ so that $\mathbb{Z}^d$ is transient. Let the vertex set $\mathcal{VG}=\mathbb{Z}^d\times\{0,1\}$. The edge set $\mathcal{EG}$ is defined as
    \[
        \mathcal{EG}:=\{[(x,y),(x',y')]\mid x=x' \text{ and } y\ne y' \text{ in } \{0,1\} \text{ or } y=y' \text{ and } x\sim x' \text{ in } \mathbb{Z}^d \}.
    \]
    Then for any finite set $A$, $\mathcal{G}\setminus A$ has only one component, so $\mathcal{G}$ is one-ended. 
    For $x\in\mathbb{Z}^d$, let $|x|$ be the graph distance on $\mathbb{Z}^d$ between $x$ and the origin. The conductance function is defined as
    \[
        \forall (x,y)\sim(x',y'),\quad \mathfrak{c}([(x,y),(x',y')]):=
        \begin{cases}
             (|x|+1)^{-(d+1)},&\quad \text{if }x=x'\\
             1,&\quad \text{if }y=y'.
        \end{cases}
    \]

    Define the cut set between $\mathbb{Z}^d\times\{0\}$ and $\mathbb{Z}^d\times\{1\}$ by
    \[
        F=\{[(x,y),(x',y')])\mid x=x'\}.
    \]
    Then
    \[
        \sum_{e\in F}\mathfrak{c}(e)\le 1+\sum_{n\ge 1}\sum_{x:|x|=n}n^{-(d+1)}
        \le 1+\sum_{n\ge 1} C_d n^{d-1}n^{-(d+1)}\le 1+C_d\sum_{n\ge 1} n^{-2}<\infty.
    \]
    
    By Corollary 1.3 from \cite{carmesin2012characterization}, \ref{HD = R} fails. 
\end{exmp}

\section*{Acknowledgments} 
This project was initiated as part of the UChicago REU program. The author would like to thank Ewain Gwynne for introducing this problem, Peter May for organizing the REU, and is also grateful to Ewain Gwynne, Xinyi Li, Yuval Peres and Wenyuan Yang for many stimulating discussions.
\bibliographystyle{plain}
\bibliography{RI_and_RWinfty}

@article{eisenbaum2022isomorphism,
      title={Isomorphism theorems, extended {M}arkov processes and random interlacements},
      author={Eisenbaum, Nathalie and Kaspi, Haya},
      journal={Electronic Journal of Probability},
      volume={27},
      year={2022},
      number = {none},
      publisher = {Institute of Mathematical Statistics and Bernoulli Society},
      pages = {1 -- 27},
      doi = {10.1214/22-EJP887},
      URL = {https://doi.org/10.1214/22-EJP887}
    }

@book {drs-interlacement-book,
      title={An introduction to random interlacements},
      author={Drewitz, Alexander and R\'{a}th, Bal\'{a}zs and Sapozhnikov, Art\"{e}m},
      series={SpringerBriefs in Mathematics},
      publisher={Springer, Cham},
      year={2014},
      pages={x+120},
      isbn={978-3-319-05851-1; 978-3-319-05852-8},
      mrclass={60G50 (60K35 82C41)},
      mrnumber={3308116},
      mrreviewer={Ingemar Kaj},
      doi={10.1007/978-3-319-05852-8},
      url={https://doi.org/10.1007/978-3-319-05852-8}
    }

@book {lyons-peres,
      title={Probability on Trees and Networks},
      author={Lyons, Russell and Peres, Yuval},
      series={Cambridge Series in Statistical and Probabilistic Mathematics},
      volume={42},
      publisher={Cambridge University Press, New York},
      year={2016},
      pages={xv+699},
      isbn={978-1-107-16015-6},
      mrclass={60C05 (05C05 05C81 28A80 60J50 60J80 60K35 82B41)},
      mrnumber={3616205},
      doi={10.1017/9781316672815},
      url={http://dx.doi.org/10.1017/9781316672815}
    }

@article{sznitman-interlacement,
      title={Vacant set of random interlacements and percolation},
      author={Sznitman, Alain-Sol},
      journal={Annals of Mathematics. Second Series},
      volume={171},
      year={2010},
      number={3},
      pages={2039--2087},
      issn={0003-486X},
      doi={10.4007/annals.2010.171.2039}
    }

@article{teixeira-interlacement,
      title={Interlacement percolation on transient weighted graphs},
      author={Teixeira, Augusto},
      journal={Electronic Journal of Probability},
      volume={14},
      year={2009},
      number={none},
      pages={1604--1628},
      doi={10.1214/EJP.v14-670},
      eprint={\arxiv{0907.0316}},
      mrclass={60K35},
      mrnumber={2525105}
    }

@article{gs-reflected-rw,
      title={Random walk reflected off of infinity, with applications to uniform spanning forests and supercritical Liouville quantum gravity},
      author={Gwynne, Ewain and Sung, Jinwoo},
      journal={ArXiv e-prints},
      year={2025},
      month={jun},
      archivePrefix={arXiv},
      eprint={\arxiv{2506.18827}},
      primaryClass={math.PR},
      adsurl={https://ui.adsabs.harvard.edu/abs/2025arXiv250618827G},
      note={arXiv:2506.18827 [math.PR]}
    }

@article{hutchcroft2018interlacements,
      title={Interlacements and the wired uniform spanning forest},
      author={Hutchcroft, Tom},
      journal={The Annals of Probability},
      volume={46},
      number={2},
      pages={1170--1200},
      year={2018},
      publisher={JSTOR}
    }

@article{bekka1997group,
      title={Group cohomology, {H}armonic functions and the {F}irst {$L^2$}-{B}etti number},
      author={Bekka,  Mohammed E.B. and Valette, Alain},
      journal={Potential Analysis},
      volume={6},
      number={4},
      pages={313--326},
      year={1997},
      publisher={Springer},
    }

@book{bekka2007kazhdan,
      title={Kazhdan’s property {(T)}},
      author={Bekka, Bachir and de La Harpe, Pierre and Valette, Alain},
      volume={11},
      place={Cambridge}, 
      series={New Mathematical Monographs},
      publisher={Cambridge University Press}, 
      year={2008}, 
      collection={New Mathematical Monographs}}

@article{Gaboriau2002,
      title={Invariants $\ell ^2$ de relations d{\textquoteright}\'equivalence et de groupes},
      author={Gaboriau, Damien},
      journal={Publications Math\'ematiques de l'IH\'ES},
      volume={95},
      year={2002},
      pages={93--150},
      language={fre},
      publisher={Institut des Hautes Etudes Scientifiques},
      zbl={1022.37002},
      url={https://www.numdam.org/item/PMIHES_2002__95__93_0/}
    }

@book{dummit2004abstract,
      title={Abstract {A}lgebra},
      author={Dummit, David Steven and Foote, Richard M.},
      volume={3},
      year={2004},
      publisher={Wiley Hoboken}
    }

@article{carmesin2012characterization,
      title={A {C}haracterization of the {L}ocally {F}inite {N}etworks {A}dmitting {N}on-{C}onstant {H}armonic {F}unctions of {F}inite {E}nergy},
      author={Carmesin, Johannes},
      journal={Potential Analysis},
      volume={37},
      number={3},
      pages={229--245},
      year={2012},
      publisher={Springer}
    }

@article{elek1998quasi,
      title={On quasi-transitive amenable graphs},
      author={Elek, G{\'a}bor and Tardos, G{\'a}bor},
      year={1998},
      journal={ArXiv e-prints},
      eprint={math/9806129},
      archivePrefix={arXiv},
      primaryClass={math.GT},
      url={https://arxiv.org/abs/math/9806129},
      note={arXiv:math/9806129 [math.GT]}
    }

\end{document}